\documentclass[11pt]{article}
\usepackage[margin=1in]{geometry} 
\usepackage{amssymb,amsfonts,amsmath,bbm,mathrsfs,stmaryrd,mathtools}
\usepackage{xcolor}
\usepackage{url}

\usepackage{enumerate}

\usepackage[colorlinks,
linkcolor=black!75!red,
citecolor=blue,
pdftitle={},
pdfproducer={pdfLaTeX},
pdfpagemode=None,
bookmarksopen=true,
bookmarksnumbered=true,
backref=page]{hyperref}

\usepackage{tikz}
\usetikzlibrary{arrows,calc,decorations.pathreplacing,decorations.markings,intersections,shapes.geometric,through,fit,shapes.symbols,positioning,decorations.pathmorphing}

\usepackage{braket}

\usepackage[amsmath,thmmarks,hyperref]{ntheorem}
\usepackage{cleveref}

\creflabelformat{enumi}{#2(#1)#3}

\crefname{section}{Section}{Sections}
\crefformat{section}{#2Section~#1#3} 
\Crefformat{section}{#2Section~#1#3} 

\crefname{subsection}{\S}{\S\S}
\AtBeginDocument{%
  \crefformat{subsection}{#2\S#1#3}%
  \Crefformat{subsection}{#2\S#1#3}%
}

\crefname{subsubsection}{\S}{\S\S}
\AtBeginDocument{%
  \crefformat{subsubsection}{#2\S#1#3}%
  \Crefformat{subsubsection}{#2\S#1#3}%
}

\theoremstyle{plain}

\newtheorem{lemma}{Lemma}[section]
\newtheorem{proposition}[lemma]{Proposition}
\newtheorem{corollary}[lemma]{Corollary}
\newtheorem{theorem}[lemma]{Theorem}

\theoremstyle{plain}
\theoremnumbering{Alph}
\newtheorem{theoremN}{Theorem}

\theoremstyle{plain}
\theorembodyfont{\upshape}
\theoremsymbol{\ensuremath{\blacklozenge}}

\newtheorem{definition}[lemma]{Definition}
\newtheorem{example}[lemma]{Example}
\newtheorem{remark}[lemma]{Remark}
\newtheorem{convention}[lemma]{Convention}
\newtheorem{notation}[lemma]{Notation}

\crefname{definition}{definition}{definitions}
\crefformat{definition}{#2definition~#1#3} 
\Crefformat{definition}{#2Definition~#1#3} 

\crefname{ex}{example}{Examples}
\crefformat{example}{#2example~#1#3} 
\Crefformat{example}{#2Example~#1#3} 

\crefname{remark}{remark}{Remarks}
\crefformat{remark}{#2remark~#1#3} 
\Crefformat{remark}{#2Remark~#1#3} 

\crefname{convention}{convention}{Conventions}
\crefformat{convention}{#2convention~#1#3} 
\Crefformat{convention}{#2Convention~#1#3} 

\crefname{notation}{notation}{Notations}
\crefformat{notation}{#2notation~#1#3} 
\Crefformat{notation}{#2Notation~#1#3} 

\crefname{table}{table}{Tables}
\crefformat{table}{#2table~#1#3} 
\Crefformat{table}{#2Table~#1#3}

\crefname{lemma}{lemma}{Lemmas}
\crefformat{lemma}{#2lemma~#1#3} 
\Crefformat{lemma}{#2Lemma~#1#3} 

\crefname{proposition}{proposition}{Propositions}
\crefformat{proposition}{#2proposition~#1#3} 
\Crefformat{proposition}{#2Proposition~#1#3} 

\crefname{corollary}{corollary}{Corollaries}
\crefformat{corollary}{#2corollary~#1#3} 
\Crefformat{corollary}{#2Corollary~#1#3} 

\crefname{theorem}{theorem}{Theorems}
\crefformat{theorem}{#2theorem~#1#3} 
\Crefformat{theorem}{#2Theorem~#1#3} 

\crefname{theoremN}{theorem}{Theorems}
\crefformat{theoremN}{#2theorem~#1#3} 
\Crefformat{theoremN}{#2Theorem~#1#3} 

\crefname{enumi}{}{}
\crefformat{enumi}{(#2#1#3)}
\Crefformat{enumi}{(#2#1#3)}

\crefname{assumption}{assumption}{Assumptions}
\crefformat{assumption}{#2assumption~#1#3} 
\Crefformat{assumption}{#2Assumption~#1#3} 

\crefname{equation}{}{}
\crefformat{equation}{(#2#1#3)} 
\Crefformat{equation}{(#2#1#3)}

\numberwithin{equation}{section}

\theoremstyle{nonumberplain}
\theoremsymbol{\ensuremath{\blacksquare}}

\newtheorem{proof}{Proof}

\newcommand\bC{{\mathbb C}}

\newcommand\bG{{\mathbb G}}

\newcommand\bP{{\mathbb P}}

\newcommand\bZ{{\mathbb Z}}

\newcommand\cE{{\mathcal E}}
\newcommand\cF{{\mathcal F}}

\newcommand\cL{{\mathcal L}}
\newcommand\cM{{\mathcal M}}
\newcommand\cN{{\mathcal N}}
\newcommand\cO{{\mathcal O}}
\newcommand\cP{{\mathcal P}}
\newcommand\cQ{{\mathcal Q}}

\newcommand\cS{{\mathcal S}}
\newcommand\cT{{\mathcal T}}

\newcommand\cV{{\mathcal V}}

\DeclareMathOperator{\id}{id}
\DeclareMathOperator{\End}{\mathrm{End}}

\def\Sec{{\rm Sec}}

\def\Hom{\operatorname {Hom}}
\def\Aut{\operatorname{Aut}}
\def\Ext{\operatorname {Ext}}

\newcommand\spr[1]{\cite[\href{https://stacks.math.columbia.edu/tag/#1}{Tag {#1}}]{stacks-project}}
\newcommand{\qedhere}{\mbox{}\hfill\ensuremath{\blacksquare}}

\title{Compactified symplectic leaves in bundle moduli spaces}
\author{Alexandru Chirvasitu}

\begin{document}

\date{}

\newcommand{\Addresses}{{
    \bigskip
    \footnotesize

    \textsc{Department of Mathematics, University at Buffalo}
    \par\nopagebreak
    \textsc{Buffalo, NY 14260-2900, USA}  
    \par\nopagebreak
    \textit{E-mail address}: \texttt{achirvas@buffalo.edu}

    
  }}

\maketitle

\begin{abstract}
  Let $\mathcal{E}$ be a rank-2 vector bundle over an elliptic curve $E$, decomposable as a sum of line bundles of degrees $d'>d\ge 2$, and $\mathcal{L}$ the determinant of $\mathcal{E}$. The subspace $L(\mathcal{E})\subset \mathbb{P}^{n-1}\cong \mathbb{P}\mathrm{Ext}^1(\mathcal{L},\mathcal{O}_E)$ consisting of classes of extensions with middle term isomorphic to $\mathcal{E}$ is one of the symplectic leaves of a remarkable Poisson structure on $\mathbb{P}^{n-1}$ defined by Feigin-Odesskii/Polishchuk, and all symplectic leaves arise in this manner, as shown in earlier work that realizes $L(\mathcal{E})$ as the base space of a principal $\mathrm{Aut}(\mathcal{E})$-fibration.

  Here, we embed $L(\mathcal{E})$ into a larger, projective base space $\widetilde{L}(\mathcal{E})$ of a principal $\mathrm{Aut}(\mathcal{E})$-fibration whose total space consists of sections of $\mathcal{E}$. The embedding realizes $L(\mathcal{E})\subset \widetilde{L}(\mathcal{E})$ as a complement of an anticanonical divisor $Y$ (one of the main results), and we give an explicit description of the normalization of $Y$ as a projective-space bundle over a projective space. For $d=2$ $\widetilde{L}(\mathcal{E})$ is one of the three Hirzebruch surfaces $\Sigma_i$, $i=0,1,2$; we determine which occurs when and hence also the cases when $L(\mathcal{E})$ is affine.

  Separately, we prove that for $d<\frac n2$ the singular locus of the secant slice $\mathrm{Sec}_{d,z}(E)\subset \mathbb{P}^{n-1}$, the portion of the $d^{th}$ secant variety of $E$ consisting of points lying on spans of $d$-tuples with sum $z\in E$, is precisely $\mathrm{Sec}_{d-2}$. This strengthens result that $L(\mathcal{E})$ is smooth, appearing in prior joint work with R. Kanda and S.P. Smith. 
\end{abstract}

\noindent {\em Key words: elliptic curve; projective space; vector bundle; pairing; Poisson structure; symplectic leaf; characteristic class; Chern class; Chow ring; multiplicative sequence; Hirzebruch surface; secant variety; smooth locus; bilinear form; non-degenerate}

\vspace{.5cm}

\noindent{MSC 2020: 14H60; 14H52; 14C17; 14C15; 19L10; 53D17; 14J42; 53D30; 15A63}

\tableofcontents

\section*{Introduction}

The present paper is an offshoot of \cite{00-leaves-2}, and fits into a common circle of ideas.

Consider a (complex) elliptic curve $E$, a line bundle $\cL$ on $E$ of degree $n\ge 3$, and the resulting \cite[Proposition II.7.3]{hrt} embedding\footnote{This makes $E\subset \bP^{n-1}$ into a {\it normal elliptic curve}, i.e. one of degree $n$ contained in no hyperplane (e.g. \cite[Introduction]{fis10}).\label{foot:norm}}
\begin{equation}\label{eq:eemb}
  E\subset \bP^{n-1}:=\bP H^0(\cL)^*\cong \bP \Ext^1(\cL,\cO),
\end{equation}
where
\begin{itemize}
\item $\cO=\cO_E$ is the trivial line bundle on $E$;
\item the {\it projectivization} $\bP V$ of a vector space $V$ is $(V\setminus\{0\})/(\text{scaling})$;
\item we identify the space $V:=H^0(\cL)=\Hom(\cO,\cL)$ of sections of $\cL$ with that of sections of the Serre {\it twisting sheaf} $\cO_{\bP V^*}(1)$ \cite[Definition preceding Proposition II.5.12]{hrt} on $\bP V^* = \bP H^0(\cL)^*$;
\item and the last isomorphism is {\it Serre duality} \cite[Theorem III.7.1]{hrt}. 
\end{itemize}

The ambient space $\bP^{n-1}$ of that embedding is equipped in the literature with a {\it Poisson structure} \cite[Definition 1.1.1]{cp_qg} in two operationally distinct ways:
\begin{enumerate}[(a)]
\item On the one hand there is the procedure familiar \cite[\S 1.6 A]{cp_qg} from {\it deformation quantization}, sketched in \cite[Introduction]{FO98}. \cite[Proposition 3]{fo89-eng} introduces a deformation $Q_{n,1}(E,\eta)$ of the symmetric algebra $SV = SH^0(\cL)$ depending on a point $\eta\in E$ (see also \cite[\S 2.III, Proposition]{FO98}), generated by $V$, with $Q_{n,1}(E,0)\cong SV$. One then obtains a Poisson bracket between elements of the generating space $V$ (i.e. between linear functionals on $V^*$) as in \cite[\S 1.6 A]{cp_qg}, by differentiating commutators as $\eta\xrightarrow{}0$, and that Poisson bracket descends to $\bP V^*$ by homogeneity. 

\item On the other hand, there is the geometric/moduli-theoretic approach of \cite[\S 2]{pl98}, which goes through once one identifies \cite[\S III.5, Exercise 2]{gm} $\bP V^*\cong\bP \Ext^1(\cL,\cO)$ with the (set of) classes of non-split extensions
  \begin{equation}\label{eq:bulletext}
    0\xrightarrow{}\cO\xrightarrow{\quad}\bullet\xrightarrow{\quad}\cL\xrightarrow{}0,
  \end{equation}
  two being declared equivalent\footnote{Allowing {\it isomorphisms} as the outer vertical arrows as opposed to identities (as is more customary: \cite[\S III.5, Exercise 2]{gm}, \cite[Theorem 7.30]{R}, etc.), together with the fact that $\End(\cL)\cong \End(\cO)\cong \bC$ \cite[Theorem I.3.4 (a)]{hrt}, is what accounts for the projectivization in $\bP\Ext^1(\cL,\cO)$.} if they fit into a commutative diagram
  \begin{equation*}
    \begin{tikzpicture}[>=stealth,auto,baseline=(current  bounding  box.center)]
      \path[anchor=base] 
      (0,0) node (ll) {$0$}
      +(1,.5) node (ul) {$\cO$}
      +(1,-.5) node (dl) {$\cO$}
      +(3,.8) node (um) {$\cE$}
      +(3,-.8) node (dm) {$\cE'$}
      +(5,.5) node (ur) {$\cL$}
      +(5,-.5) node (dr) {$\cL$}
      +(6,0) node (rr) {$0$.}
      ;
      \draw[->] (ll) to[bend left=6] node[pos=.5,auto] {$\scriptstyle $} (ul);
      \draw[->] (ll) to[bend right=6] node[pos=.5,auto] {$\scriptstyle $} (dl);
      \draw[->] (ul) to[bend left=6] node[pos=.5,auto] {$\scriptstyle $} (um);
      \draw[->] (dl) to[bend right=6] node[pos=.5,auto] {$\scriptstyle $} (dm);
      \draw[->] (um) to[bend left=6] node[pos=.5,auto] {$\scriptstyle $} (ur);
      \draw[->] (dm) to[bend right=6] node[pos=.5,auto] {$\scriptstyle $} (dr);
      \draw[->] (ur) to[bend left=6] node[pos=.5,auto] {$\scriptstyle $} (rr);
      \draw[->] (dr) to[bend right=6] node[pos=.5,auto] {$\scriptstyle $} (rr);
      \draw[->] (ul) to[bend right=0] node[pos=.5,auto] {$\scriptstyle \cong$} (dl);
      \draw[->] (um) to[bend right=0] node[pos=.5,auto] {$\scriptstyle \cong$} (dm);
      \draw[->] (ur) to[bend right=0] node[pos=.5,auto,swap] {$\scriptstyle \cong$} (dr);
    \end{tikzpicture}
  \end{equation*}
\end{enumerate}
The first proof that the two Poisson structures coincide seems to have appeared only much later, as \cite[Theorem 5.2]{HP1}.

Given this single remarkable Poisson structure, the natural follow-up problem is to describe its {\it symplectic leaves} \cite[Definition-Proposition 1.1.2]{cp_qg} (roughly speaking, maximal immersed submanifolds on which the Poisson structure restricts to a symplectic structure). \cite[Theorem 1]{FO98} (or rather a particular case thereof) states that these are what \cite[\S 1.2]{00-leaves-2} refers to as the {\it homological leaves} attached to rank-2 vector bundles $\cE$ on $E$:
\begin{equation*}
  L(\cE):=\left\{\text{classes of extensions \Cref{eq:bulletext} with }\bullet\cong \cE\right\}
\end{equation*}
(so that in particular $\cL\cong \wedge^2 \cE$, the {\it determinant} \cite[Exercise II.6.11]{hrt} of the rank-2 vector bundle $\cE$). \cite[Theorem 1.6]{00-leaves-2} proves that this is indeed the case, by realizing $L(\cE)$ as a {\it geometric quotient} \cite[Definition 0.6]{mumf-git}
\begin{equation*}
  X(\cE)\xrightarrow{\quad}L(\cE)\cong X(\cE)/\Aut(\cE).
\end{equation*}
Here, $X(\cE)\subset H^0(\cE)$ consists of those sections $s$ fitting into an exact sequence
\begin{equation*}
  0\xrightarrow{} \cO\xrightarrow{\ s\ }\cE\xrightarrow{\quad}\cL\xrightarrow{}0
\end{equation*}
(equivalently: $\mathrm{coker}~s$ is torsion-free) and the map onto $L(\cE)$ is what one might guess:
\begin{equation}\label{eq:xe2le}
  X(\cE)\ni s
  \xmapsto{\quad}
  \left(\text{class of}\quad 0\xrightarrow{}\cO\xrightarrow{s}\cE\xrightarrow{}\cL\xrightarrow{}0\right)
  \in L(\cE).
\end{equation}

More is true:
\begin{itemize}
\item The action of $\Aut(\cE)$ on $X(\cE)$ is \cite[Proposition 5.11]{00-leaves-2} {\it principal}\footnote{This is the appropriate notion of freeness for actions of linear algebraic groups on varieties; the condition is generally strictly stronger \cite[Example 0.4]{mumf-git} than set-theoretic freeness.} \cite[\S 1.8, following Corollary]{borel-LAG} the map
  \begin{equation*}
    X(\cE)\times \Aut(\cE)
    \xrightarrow{\text{(1$^{st}$ projection, action)}}
    X(\cE)\times X(\cE)
  \end{equation*}
  is a scheme isomorphism. 

\item And in fact \Cref{eq:xe2le} is \cite[Theorem 5.14]{00-leaves-2} a {\it locally trivial} \cite[Expos\'e XI, Remarque 4.7]{sga1} principal $\Aut(\cE)$-fibration, in the Zariski topology and hence a fortiori also in the standard topology.
\end{itemize}

The point of departure for the present work is the observation that $X(\cE)\subset H^0(\cE)$ is not, in general, the largest natural domain for a free $\Aut(\cE)$-action: setting
\begin{equation*}
  H^0(\cE)\supseteq \widetilde{X}(\cE)
  :=
  \left\{\text{sections with trivial isotropy group in }\Aut(\cE)\right\}\supseteq X(\cE),
\end{equation*}
it is a simple remark that $\Aut(\cE)$ again acts principally on $\widetilde{X}(\cE)$. Furthermore, for the most interesting class of rank-2 bundles $\cE$ for the present purpose, decomposable as
\begin{equation*}
  \cE\cong \cN\oplus \cN',\quad 2\le \deg\cN\le \deg\cN'
\end{equation*}
for line bundles $\cN$ \cite[Theorem 1.1]{00-leaves-2}, the quotient $\widetilde{L}(\cE):=\widetilde{X}(\cE)/\Aut(\cE)$ (the {\it complete homological leaf} attached to $\cE$: \Cref{def:projgit}) is in fact {\it projective}, and the right-hand embedding in   
\begin{equation*}
  \begin{tikzpicture}[>=stealth,auto,baseline=(current  bounding  box.center)]
    \path[anchor=base] 
    (0,0) node (l) {$X(\cE)$}
    +(2,.5) node (u) {$\widetilde{X}(\cE)$}
    +(2,-.5) node (d) {$L(\cE)$}
    +(4,0) node (r) {$\widetilde{L}(\cE)$}
    ;
    \draw[right hook->] (l) to[bend left=6] node[pos=.5,auto] {$\scriptstyle $} (u);
    \draw[->>] (u) to[bend left=6] node[pos=.5,auto] {$\scriptstyle $} (r);
    \draw[->>] (l) to[bend right=6] node[pos=.5,auto,swap] {$\scriptstyle $} (d);
    \draw[right hook->] (d) to[bend right=6] node[pos=.5,auto,swap] {$\scriptstyle $} (r);
  \end{tikzpicture}
\end{equation*}
is a divisor complement. It was this that originally motivated the investigation, given that one can often extract additional information on varieties by virtue of their realization as complements of effective divisors in projective varieties: complements of effective {\it ample} divisors are affine \cite[Introduction]{goodman-affine}, much information is available on divisor-complement fundamental groups via Lefschetz-type theorems (\cite[Corollary 2.10]{nori_zar}, \cite[Theorem and Corollary]{shim_homog}), and so on.

For that reason, it seemed worthwhile to try to better understand ``how'' $L(\cE)$ embeds into $\widetilde{L}(\cE)$. In that direction, and rendered somewhat roughly in paraphrase, \Cref{th:ynorm} (supplemented by \Cref{pr:anticfinal}) reads as follows.

\begin{theoremN}\label{th:introA}
  Let $\cE\cong \cN\oplus \cN'$ be a rank-2 vector bundle on $E$, for line bundles $\cN$ and $\cN'$ with $\deg\cN'>\deg\cN\ge 2$.

  The complement
  \begin{equation*}
    Y:=\widetilde{L}(\cE)\setminus L(\cE)
  \end{equation*}
  is then an effective anticanonical divisor, whose normalization we describe explicitly as a projective-space bundle over a projective space.

  Furthermore, that normalization is an isomorphism (and hence $Y$ is normal) precisely when $d=2$.  \qedhere
\end{theoremN}

In that latter case, when the line bundle $\cN$ has degree {\it precisely} 2, much more can be said (\Cref{pr:homishirz,th:isantican,th:d2hirz}), including (as hinted at above) a complete enumeration of the cases when the homological leaf $L(\cE)$ is affine. 

\begin{theoremN}\label{th:introB}
  Let $\cE\cong \cN\oplus \cN'$ be a rank-2 vector bundle on $E$, for line bundles $\cN$ and $\cN'$ with $\deg\cN'>\deg\cN = 2$.
  \begin{enumerate}[(1)]
  \item The complete homological leaf $\widetilde{L}(\cE)$ is one of three {\it Hirzebruch surfaces} \cite[\S V.4]{bhpv}
    \begin{equation*}
      \Sigma_e:=\text{$\bP^1$-bundle over $\bP^1$ associated to the vector bundle }\cO\oplus \cO(e):
    \end{equation*}
    \begin{itemize}
    \item $\Sigma_1$ when $\deg\cN'$ is odd;
    \item $\Sigma_2$ when $\deg\cN'$ is even and $\cN'$ is a tensor power of $\cN$;
    \item and $\Sigma_0\cong \bP^1\times \bP^1$ otherwise. 
    \end{itemize}

  \item The homological leaf $L(\cE)\subset \widetilde{L}(\cE)$ is the complement of an anticanonical divisor.

  \item That complement is affine when $\deg\cN'$ is odd.

  \item While for even $\deg\cN'$ the following conditions are equivalent:
    \begin{itemize}
    \item $L(\cE)$ is affine;
    \item $L(\cE)$ is quasi-affine;
    \item the anticanonical divisor $\widetilde{L}(\cE)\setminus L(\cE)$ is (very) ample;
    \item $\cN'$ is not a tensor power of $\cN$.  \qedhere
    \end{itemize}
  \end{enumerate}
\end{theoremN}

\Cref{se:pair} is meant as preliminary. The goal is to discuss a general framework (\Cref{def:pdpd}) for constructing vector bundles $\cS_{\beta}$ and $\cQ_{\beta}$ over $\bP V$, of ranks $\dim W$ and $\dim V'-\dim W$ respectively, starting with a bilinear map
\begin{equation*}
  V\otimes W\xrightarrow{\quad \beta\quad} V'
\end{equation*}
that restricts to embeddings when either variable is fixed to a non-zero vector. This specializes back to the preceding discussion by taking for $\beta$ the multiplication
\begin{equation*}
  H^0(\cN)\otimes H^0(\cN'\otimes \cN^{-1})
  \xrightarrow{\quad}
  H^0(\cN')
\end{equation*}
for line bundles $\cN$ and $\cN'$ as in \Cref{th:introA,th:introB}. In that case, the complete leaf $\widetilde{L}(\cE)$ is identifiable (\Cref{th:wtlispp}) with the projectivization of $\cQ_{\beta}$ and the normalization left nebulous in \Cref{th:introA} can similarly be described in terms of said bundles (\Cref{th:ynorm} \Cref{item:ynorm-imy} and \Cref{item:ynorm-norm}). 

The more general material consists of a number of examples, Chern-class computations and structure results (\Cref{pr:chern,pr:mustsurj}, \Cref{cor:whenhavetriv}) handy later in the specific applications to homological leaves, and also perhaps of some independent interest.

\Cref{se:codim3} makes a lateral move to the adjacent (albeit related) topic of determining singular loci for {\it secant varieties} (or rather close cousins thereof). Recall (\cite[Example 8.5]{H92}, \cite[\S V.1]{zak_sec}) that the $d^{th}$ secant variety
\begin{equation*}
  \Sec_d(E)\subseteq \bP^{n-1}
\end{equation*}
of a normal elliptic curve \Cref{eq:eemb} is the closure of set of points lying on the span of generic $d$-tuples in $E$. Restricting only to those $d$-tuples whose sum (with respect to the group law on $E$) is a fixed $z\in E$, one obtains the ``slices'' $\Sec_{d,z}(E)\subseteq \bP^{n-1}$ (see \Cref{eq:secdz} and surrounding discussion, as well as \cite[\S 1.3]{00-leaves-2}).

\cite[Theorem 1.2]{00-leaves-2} effects the link to the preceding discussion: for a bundle $\cE$ as in \Cref{th:introA,th:introB} we have 
\begin{equation*}
  L(\cE) = \Sec_{\deg\cN,z}\setminus \Sec_{\deg\cN-1},
\end{equation*}
where $z$ is the sum over the divisor of zeros of any section $0\ne s\in H^0(\cN)$. Since $L(\cE)$ is also smooth \cite[Theorem 1.5]{00-leaves-2}, this brings up the question of determining the singular loci of the aforementioned secant slices. For secant varieties proper (assuming $d<\frac n2$, say) $\Sec_d$ is smooth precisely off $\Sec_{d-1}$ (\cite[Proposition 8.15]{gvb-hul} or \cite[discussion following Theorem 1.4]{fisher2006} treat explicitly only the difficult implication, \cite[Theorem]{copp} proves a generalization of the full claim, etc.). The slice analogue is \Cref{th:singlocdz}:

\begin{theoremN}
  Given a point $z\in E$ on a normal elliptic curve $E\subset \bP^{n-1}$ and a positive integer $d<\frac n2$, the singular locus of $\Sec_{d,z}$ is precisely $\Sec_{d-2}$.  \qedhere
\end{theoremN}

\subsection*{Acknowledgements}

I am grateful for valuable input from R. Kanda, S.P. Smith and M. Fulger.

This work is partially supported through NSF grant DMS-2001186. 


\section{Pairings and vector bundles on projective spaces}\label{se:pair}

The construction discussed in the present section associates a vector bundle on a projective space to a ``sufficiently non-degenerate'' pairing between two finite-dimensional vector spaces. First, a simple remark that is perhaps worth setting out for future reference:

\begin{lemma}\label{le:vv'}
  Let $V$ and $V'$ be two finite-dimensional vector spaces with
  \begin{equation*}
    d:=\dim V
    \quad\text{and}\quad
    d+k:=\dim V',\ k\ge 0.
  \end{equation*}
  For a non-negative integer $\ell$, consider the Grassmannian
  \begin{equation*}
    \bG_{\ell}=\bG_{\ell,V,V'}:=\bG(\ell,\bP \Hom(V,V'))
  \end{equation*}
  of $\ell$-planes in the respective projective space.
  \begin{enumerate}[(1)]
  \item The subspace $X\subset \bG_{\ell}$ consisting of those $\ell$-planes all of whose elements are (classes of) {\it embeddings} $V\le V'$ is an open subvariety.
  \item And furthermore, that subvariety is non-empty (and hence dense) if and only if
    \begin{equation*}
      \ell\le k = \dim V'-\dim V.
    \end{equation*}
  \end{enumerate}
\end{lemma}
\begin{proof}
  We obtain $X$ by eliminating the $\ell$-planes containing (the line through) at least one non-injective morphism $V\to V'$. Or, to say it differently: if
  \begin{equation}\label{eq:noninj}
    Y\subset \bP \Hom(V,V')
  \end{equation}
  is the set of lines through non-injections, then $\bG\setminus X$ is the set of $\ell$-planes meeting $Y$; in common parlance \cite[Example 6.14]{H92}, $\bG\setminus X$ is the {\it variety of incident planes} to $Y$. Since $Y$ is a closed variety (being definable by minor vanishing), $\bG\setminus X\subset \bG$ too must be a closed variety by the selfsame \cite[Example 6.14]{H92}.

  As for the dimension claim, one implication is obvious by example: identify $V\cong S^{d-1}\bC^2$ and $V'\cong S^{d+k-1}\bC^2$ (symmetric powers), and consider the multiplication
  \begin{equation*}
    V\otimes S^k\bC^2\to V'. 
  \end{equation*}
  Every non-zero element of the $(k+1)$-dimensional vector space $L:=S^k\bC^2$ will operate as an injection $V\le V'$, so the projectivization of $L$ will be a $k$-plane belonging in $X\subset \bG$.

  Conversely, we have to argue that all $(k+1)$-planes in $\bP\Hom(V,V')$ intersects the non-injection variety \Cref{eq:noninj}. To see this, note first that $Y$ is the closure of the space of lines through morphisms $V\to V'$ of maximal non-full rank, i.e. $\dim V-1=d-1$. Specifying such a morphism amounts to
  \begin{itemize}
  \item fixing a $(d-1)$-dimensional subspace of $V'$, and hence a point in a $(d-1)(k+1)$-dimensional Grassmannian;
  \item and also, up to scaling, a surjection from $V$ onto that space, i.e. an element in an open subvariety of a $d(d-1)$-dimensional affine space.
  \end{itemize}
  All in all, then,
  \begin{equation*}
    \dim Y = (d-1)(k+1) + d(d-1) -1  = (d-1)(d+k+1)-1.
  \end{equation*}
  Its codimension in the $(d(d+k)-1)$-dimensional projective space $\bP\Hom(V,V')$ is thus 
  \begin{equation*}
    d(d+k) - (d-1)(d+k+1) = k+1,
  \end{equation*}
  so that indeed every $(k+1)$-plane in $\bP\Hom(V,V')$ will intersect $Y$ \cite[Theorem I.7.2]{hrt}.
\end{proof}

The following projective bundles over projective spaces will play an essential role below.

\begin{definition}\label{def:pdpd}
  Consider a bilinear map
  \begin{equation}\label{eq:bilin}
    V\otimes W\xrightarrow{\quad \beta\quad} V'
  \end{equation}
  for finite-dimensional vector spaces $V$, $V'$ and $W$.
  \begin{enumerate}[(1)]
  \item $\beta$ is {\it strongly non-degenerate} (or {\it 1-generic} \cite[p.541]{eis-det}) if $\beta(-,w)$ is an embedding $V\le V'$ for every non-zero $w\in W$.

    Note that the definition is automatically symmetric: the condition implies also that
    \begin{equation*}
      \beta(v,-):W\to V',\ 0\ne v\in V
    \end{equation*}
    are embeddings, and vice versa. By \Cref{le:vv'}, strongly non-degenerate pairings exist if and only if
    \begin{equation*}
      \dim V+\dim W\le \dim V'+1.
    \end{equation*}

  \item Given a strongly non-degenerate bilinear form \Cref{eq:bilin}, the {\it subbundle $\cS_{\beta}$ attached (or associated)} to it has 
    \begin{itemize}
    \item the projective space $\bP V$ of lines in $V$ as its base;
    \item as its fiber over the line $\mathrm{span}\{v\}\in \bP V$ the vector space $\mathrm{Im}~\beta(v,-)$.
    \end{itemize}
    $\cS_{\beta}$ is a vector bundle over the projective space $\bP V\cong \bP^{\dim V-1}$, with fiber $\bC^{\dim W}$.
    
  \item Similarly, the {\it quotient bundle $\cQ_{\beta}$ attached (or associated)} to $\beta$ is
    \begin{equation*}
      \cQ_{\beta}\cong \cO\otimes V'/\cS_{\beta},
    \end{equation*}
    having identified $\cS_{\beta}$ with a subbundle of the trivial bundle $\cO\otimes V'$ in the obvious fashion

    $\cQ_{\beta}$ is a vector bundle over the projective space $\bP V\cong \bP^{\dim V-1}$, with fiber $\bC^{\dim V'-\dim W}$.
    
  \item The {\it projective bundle $\bP_{\beta}$ attached (or associated)} to $\beta$ is the projectivization $\bP \cQ_{\beta}$.
  \end{enumerate}
  Since, strictly speaking, specifying $\beta$ does not distinguish between $V$ and $W$, one would also have to specify which of the two provides the projective-space base (i.e. clarify that it is $\bP V$ that supports the vector bundle $\cQ_{\beta}$). In order to keep the notation streamlined, we will do this positionally: given \Cref{eq:bilin}, the projectivization of the {\it left}-hand tensorand $V$ is the base of $\cQ_{\beta}$.  
\end{definition}

In the sequel, we will be particularly interested in the case $\dim V+\dim W=\dim V'$, so that $\bP_{\beta}$ is a $\bP^{d-1}$-bundle over $\bP^{d-1}$ for $d:=\dim V$.

\begin{remark}\label{re:pullquot}
  A different way of phrasing the construction in \Cref{def:pdpd}: consider the morphism
  \begin{equation*}
    \bP V\ni \mathrm{span}\{v\}
    \xmapsto{\quad\varphi_{\beta}\quad}
    \mathrm{Im}~\beta(v,-)\in G(\dim W,V') = \bG(\dim W-1,\bP V')=:\bG,
  \end{equation*}
  the Grassmannian of $(\dim W)$-dimensional subspaces of $V'$ (notation as in \cite[\S 3.2]{3264}). $\cQ_{\beta}$ is then the pullback $\varphi_{\beta}^*\cQ$ through $\varphi_{\beta}$ of the {\it universal quotient bundle} $\cQ$ on $\bG$ \cite[\S 3.2.3]{3264}.  
\end{remark}

The {\it characteristic classes} (in algebraic-topology parlance: e.g. \cite{ms-cc}) of the bundle $\cQ_{\beta}$ are not difficult to understand. To make sense of \Cref{pr:chern}, recall The {\it total Chern class}
  \begin{equation*}
    c(\cQ_{\beta}) = 1+c_1(\cQ_{\beta})+c_2(\cQ_{\beta})+\cdots
  \end{equation*}
  of $\cQ_{\beta}$ (\cite[\S 3.2]{Fulton-2nd-ed-98}, \cite[Theorem 5.3]{3264}, \cite[\S A.3]{hrt}, etc.): an element of the {\it Chow ring} $A(\bP V)$ of \cite[\S A.1]{hrt} (or {\it intersection ring} of \cite[\S 8.3]{Fulton-2nd-ed-98}).

\begin{proposition}\label{pr:chern}
  Fix a strongly non-degenerate $\beta$ as in \Cref{eq:bilin}. The total Chern class $c(\cQ_{\beta})$ is
  \begin{equation*}
    c(\cQ_{\beta}) = (1+\zeta+\zeta^2+\cdots)^{\dim W}
    \in
    A(\bP V) \cong \bZ[\zeta]/(\zeta^{\dim V}),\quad \zeta=c_1(\cO_{\bP V}(1)).
  \end{equation*}
\end{proposition}
\begin{proof}
  The description of $A(\bP V)$ recalled in passing is \cite[Theorem 2.1]{3264}. As for the main claim on $c(\cQ_{\beta})$, observe first that each non-zero vector $w\in W$ provides an embedding $\beta(-,w):V\to V'$ gluing, for fixed $0\ne w\in W$, to an embedding
  \begin{equation*}
    0\to \cO_{\bP V}(-1)\to \cO_{\bP V}\otimes V'.
  \end{equation*}  
  A decomposition of $W$ as a direct sum of lines then exhibits $\cQ_{\beta}$ as the third term in an exact sequence
  \begin{equation}\label{eq:embvb}
    0\to \cO_{\bP V}(-1)^{\dim W}\to \cO_{\bP V}\otimes V'\to \cQ_{\beta}\to 0.
  \end{equation}
  The multiplicativity of the total Chern class $c(-)$ with respect to exact sequences \cite[Theorem 5.3 (c)]{3264} and the fact that $c(-)=1$ on trivial bundles then ensure that
  \begin{equation}\label{eq:chernvb}
    c(\cQ_{\beta})
    = c(\cO_{\bP V}(-1))^{-\dim W}
    = (1-\zeta)^{-\dim W}
     = (1+\zeta+\zeta^2+\cdots)^{\dim W},
   \end{equation}
   finishing the proof.
\end{proof}

\begin{remark}\label{re:chern-d2}
  In particular, when $\dim V=2$ (so that $\bP V\cong \bP^1$) the degree of the rank-2 bundle $\cQ_{\beta}$ is always the coefficient of $\zeta$ in \Cref{eq:chernvb}, i.e. $\dim W$, and hence
  \begin{equation}\label{eq:manyds}
    \cQ_{\beta}\cong \cO(d_1)\oplus\cO(d_2)\oplus\cdots\oplus \cO(d_{\dim V'-\dim W})
    \quad\text{with}\quad
    \sum_i d_i=\dim W.
  \end{equation}
  This will be of some use below.  
\end{remark}

Noted for future reference:

\begin{lemma}\label{le:allnonneg}
  For a strongly non-degenerate pairing \Cref{eq:bilin} with $\dim V=2$ we have
  \begin{equation}\label{eq:allnonneg}
     \cQ_{\beta}\cong \bigoplus_{i=1}^{\dim V'-\dim W}\cO(d_i)
     \ \text{with}\ 
     \sum_i d_i=\dim W,\ d_i\ge 0.
   \end{equation}
\end{lemma}
\begin{proof}
  Only the non-negativity of the $d_i$ in \Cref{eq:manyds} needs (a little) work.

  Because $H^i(\cO(-1))$, $i=0,1$ vanish on every projective space \cite[p.4]{oss-vb}, the long exact cohomology sequence resulting from \Cref{eq:embvb} identifies the spaces of sections of the last two terms:
  \begin{equation}\label{eq:samev'}
    H^0(\cQ_{\beta})\cong V'\Longrightarrow \dim H^0(\cQ_{\beta}) = \dim V'. 
  \end{equation}      
  The section space of the generic term $\cO(d_i)$ of \Cref{eq:manyds} is $(d_i+1)$-dimensional if $d_i\ge 0$ and trivial otherwise (\cite[p.4]{oss-vb} again), so \Cref{eq:manyds} and \Cref{eq:samev'} jointly ensure that in fact all $d_i$ in \Cref{eq:manyds} are indeed non-negative
\end{proof}

The relevance of \Cref{def:pdpd} to the paper stems from the fact that the complete homological leaves $\widetilde{L}(\cE)$ of \Cref{def:projgit} will turn out to be precisely such varieties $\bP_{\beta}$. A few examples will help get a sense of the construction.

\begin{example}\label{ex:smallest}
  Suppose $\dim W=1$ (notation as in \Cref{eq:bilin}). Specifying $\beta$ then simply means giving an embedding $V\le V'$. Upon splitting that embedding arbitrarily as
  \begin{equation*}
    V'\cong V\oplus \bC^{\oplus(\dim V'-\dim V)},
  \end{equation*}
  the vector bundle $\cQ_\beta$ correspondingly splits as
  \begin{equation*}
    \cQ_{\beta}\cong \cT_{\bP V}(-1)\oplus \cO_{\bP V}^{\oplus (\dim V'-\dim V)},
  \end{equation*}
  where $\cT_{\bP V}$ denotes the tangent bundle. To see this, recall the {\it Euler sequence}
  \begin{equation*}
    0\to \cO_{\bP^d}(-1)\to \cO_{\bP^d}^{\oplus (d+1)}\to \cT_{\bP^d}(-1)\to 0
  \end{equation*}
  of \cite[p.3 equation (2)]{oss-vb}, which ensures that $\cT_{\bP V}(-1)$ is the universal quotient bundle of that projective space.  
\end{example}

\begin{example}\label{ex:tensor}
  Now take for $\beta$ the identity map
  \begin{equation*}
    V\otimes W\xrightarrow{\quad \id\quad} V\otimes W =: V'. 
  \end{equation*}
  A decomposition of $W$ as a direct sum of lines splits this pairing as a direct sum of copies of \Cref{ex:smallest} (with $V'\cong V$), and hence
  \begin{equation*}
    \cQ_{\beta} = \cQ_{\id}\cong \cT_{\bP V}(-1)^{\oplus \dim W}. 
  \end{equation*}  
\end{example}

\begin{remark}\label{re:splitbeta}
  A few other simple observations:
  \begin{enumerate}[(1)]
  \item\label{item:splitbeta} More generally, if $\beta$ decomposes as a direct sum $\beta\cong \beta_0\oplus \beta_1$ for strongly non-degenerate
    \begin{equation*}
      \beta_i:V\otimes W_i\to V'_i,\quad i=0,1
    \end{equation*}
    then we have
    \begin{equation*}
      \cQ_{\beta}\cong \cQ_{\beta_0}\oplus \cQ_{\beta_1}.
    \end{equation*}
  \item\label{item:surjbeta} On the other hand, if $\beta$ factors as a strongly non-degenerate $\beta_0$ taking values in a subspace $V_0'\le V'$, then splitting the inclusion $V'_0\le V'$ will provide an isomorphism
    \begin{equation*}
      \cQ_{\beta}\cong \cQ_{\beta_0}\oplus \cO_{\bP V}^{\oplus (\dim V'-\dim V'_0)}.
    \end{equation*}
  \end{enumerate}
  The reader will have no difficulties verifying these two claims.  
\end{remark}

The following result supplies a converse to \Cref{re:splitbeta} \Cref{item:splitbeta}.

\begin{proposition}\label{pr:mustsurj}
  Let $\beta$ be a strongly non-degenerate pairing as in \Cref{eq:bilin}.
  \begin{enumerate}[(a)]
  \item\label{item:bijepi} There is a bijection between non-zero functionals $V'\to \bC$ that annihilate $\mathrm{Im}~\beta$ and epimorphisms $\cQ_{\beta}\to \cO_{\bP V}$.
  \item\label{item:exepi} In particular, an epimorphism $\cQ_{\beta}\to \cO_{\bP V}$ exists if and only if $\beta$ is not onto.
  \item\label{item:autosplit} Epimorphisms $\cQ_{\beta}\to \cO_{\bP V}$ are automatically split. 
  \end{enumerate}    
\end{proposition}
\begin{proof}
  Part \Cref{item:exepi} is of course a consequence of \Cref{item:bijepi}; for the moment, we focus on the latter.

  One direction is (essentially) already described in \Cref{re:splitbeta} \Cref{item:surjbeta}, modulo different phrasing: an onto functional $V'\to \bC$ annihilating $\mathrm{Im}~\beta$ sheafifies to an epimorphism
  \begin{equation*}
    \cO_{\bP V}\otimes V'\twoheadrightarrow \cO_{\bP V}
  \end{equation*}
  which annihilates the leftmost term in \Cref{eq:embvb}, and hence descends to an epimorphism
  \begin{equation*}
    \cO_{\bP V}\otimes V'/\cO_{\bP V}(-1)^{\oplus \dim W}\cong \cQ_{\beta}\to \cO_{\bP V}.
  \end{equation*}

  Conversely, given \Cref{eq:embvb}, an epimorphism $\cQ_{\beta}\twoheadrightarrow \cO_{\bP V}$ is nothing but an epimorphism
  \begin{equation*}
    \cO_{\bP V}\otimes V'\xrightarrow{\ \varphi\ } \cO_{\bP V}
  \end{equation*}
  annihilating the leftmost term $\cF:=\cO(-1)^{\oplus\cdots}$ of \Cref{eq:embvb}. $\varphi$ is the sheafification of a surjection $V'\to \bC$, and the fact that it annihilates $\cF$ implies that $\beta$ factors through
  \begin{equation*}
    \ker\left(V'\to \bC\right). 
  \end{equation*}
  This disposes of part \Cref{item:bijepi}, and \Cref{item:autosplit} follows from the fact that a surjection $V'\to \bC$ splits, and hence \Cref{re:splitbeta} \Cref{item:surjbeta} applies.
\end{proof}

An immediate consequence of \Cref{pr:mustsurj}:

\begin{corollary}\label{cor:whenhavetriv}
  Let $\beta$ be a strongly non-degenerate pairing as in \Cref{eq:bilin} and suppose $\dim V=2$.

  The number of trivial terms $\cO_{\bP V}=\cO_{\bP^1}$ in the decomposition \Cref{eq:allnonneg} equals the codimension of $\mathrm{Im}~\beta$ in $V'$. In particular, such terms exist precisely when $\beta$ fails to be onto.  \qed
\end{corollary}

\begin{example}\label{ex:d21}
  Suppose
  \begin{equation*}
    \dim V=2,
    \quad
    d':=\dim V'\ge 2
    \quad\text{and}\quad
    \dim W = d'-1. 
  \end{equation*}
  A strongly non-degenerate pairing exists by \Cref{le:vv'}: this is exactly the critical case, i.e. the one where $V'$ has minimal dimension (given $\dim V$ and $\dim W$). 

  Each line in $\bC w\le W$ gives, via $\beta$, an embedding $V\le V'$ with image $\mathrm{Im}~\beta(-,w)$. Associating
  \begin{equation*}
    \bC\beta(v,w)\in \bP V'
  \end{equation*}
  to $\bC v\in \bP V$ produces a copy of the universal subbundle $\cO_{\bP^1}(-1)$ of $\bP^1\cong \bP V$ \cite[\S 3.2.3]{3264}, and writing $W$ as a direct sum of $d'-1$ lines then recovers the line bundle $\cQ_{\beta}$ on $\bP^1$ as fitting into an exact sequence
  \begin{equation}\label{eq:ses}
    0\to \cO_{\bP^1}(-1)^{\oplus (d'-1)}\to V'\otimes \cO_{\bP^1}\to \cQ_{\beta}\to 0.
  \end{equation}
  Simply taking determinants (i.e. top exterior powers) and using the multiplicativity of determinants over short exact sequences \cite[Exercise II.5.16 (d)]{hrt} and the fact that $\cQ_{\beta}$ is a line bundle, this gives
  \begin{equation*}
    \cQ_{\beta}\cong \det\cQ_{\beta}\cong \det\left(\cO_{\bP^1}(-1)^{\oplus (d'-1)}\right)^{-1}\cong \cO_{\bP^1}(d'-1)\cong \cO_{\bP^1}(\dim W). 
  \end{equation*}
  
\end{example}

\begin{example}\label{ex:hirz}
  Now set
  \begin{equation*}
    \dim V=2,
    \quad
    d':=\dim V'> 2
    \quad\text{and}\quad
    \dim W = d'-2 
  \end{equation*}
  and assume $\beta$ factors through a strongly non-degenerate
  \begin{equation*}
    V\otimes W\xrightarrow{\quad \beta_0\quad} V'_0,
    \quad
    V'_0\le V'\text{ of codimension 1}.
  \end{equation*}
  From \Cref{re:splitbeta} we obtain a splitting
  \begin{equation*}
    \cQ_{\beta}\cong \cQ_{\beta_0}\oplus \cO_{\bP^1}\cong \cO_{\bP^1}(\dim W)\oplus \cO_{\bP^1},
  \end{equation*}
  where the last isomorphism makes implicit use of \Cref{ex:d21}.

  The ruled surface
  \begin{equation*}
    \bP_{\beta} = \bP \cQ_{\beta}\to \bP^1
  \end{equation*}
  is thus the {\it $(\dim W)^{th}$ Hirzebruch surface} $\Sigma_{\dim W}$ of, say, \cite[discussion preceding Proposition V.4.2]{bhpv}.  
\end{example}

\begin{remark}\label{re:jump}
  A potentially confusing point is perhaps worth noting. Suppose we have fixed the $V$, $W$ and $V'$ of \Cref{eq:bilin}, of dimensions that will ensure the existence of strongly non-degenerate $\beta:V\otimes W\to V'$. Such $\beta$ will then form a smooth, connected, non-empty variety by \Cref{le:vv'}. Given the discreteness of the moduli space of vector bundles on $\bP^1$ (direct sums of twisting sheaves $\cO(n)$, $n\in \bZ$ \cite[Theorem 2.1.1]{oss-vb}), one might be tempted to believe that
  \begin{equation*}
    \cQ_{\beta},\quad \beta:V\to W\to V'\text{ strongly non-degenerate}
  \end{equation*}
  all admit the same decomposition
  \begin{equation}\label{eq:ods}
    \cO(d_1)\oplus \cO(d_2)\oplus \cdots
  \end{equation}
  for fixed $V$, $W$ and $V'$ with $\dim V=2$. This is not so, as the examples above show:
  \begin{itemize}
  \item for the obvious identification
    \begin{equation*}
      \bC^2\otimes \bC^2\cong \bC^4
    \end{equation*}
    the corresponding $\cV$ decomposes as $\cO(1)^{\oplus 2}$ by \Cref{ex:tensor}, since $\cT_{\bP^1}\cong \cO_{\bP^1}(2)$ (\cite[Example II.8.20.1]{hrt} or \cite[p.3 equation (2)]{oss-vb});
  \item while on the other hand, for a strongly non-degenerate pairing $\bC^2\otimes \bC^2\to \bC^4$ factoring through a subspace $\bC^3\le \bC^4$ we have $\cV\cong \cO(2)\oplus \cO$ by \Cref{ex:hirz}.
  \end{itemize}
  This is a familiar phenomenon: $\Ext^1(\cO(2),\cO)$ parametrizes extensions of $\cO(2)$ by $\cO$, generically non-split and hence isomorphic to $\cO(1)^{\oplus 2}$ \cite[Caution V.2.15.1]{hrt}, but split and isomorphic to $\cO(2)\oplus \cO$ at the origin.

  The issue is that, as noted in the discussion following \cite[Theorem 2.1.1]{oss-vb}, the one topological invariant of \Cref{eq:ods} is the first Chern class, identifiable with the sum $d_1+d_2+\cdots$. While keeping that sum fixed, degeneracies and jumps can occur in the individual $d_i$s along continuous (flat, etc.) families of bundles.  
\end{remark}

\section{Leaf completions}\label{se:ambient}

Taking a cue from \cite[Example 5.16]{00-leaves-2}, we study {\it projective} (as opposed to quasi-projective) varieties $\widetilde{L}(\cE)$ naturally housing the homological leaves $L(\cE)$ as open dense subvarieties (as divisor complements, in fact). We will be concerned exclusively with decomposable $\cE$, always within the scope of \cite[\S 2.2]{00-leaves-2}:
\begin{equation}\label{eq:splite}
  \cE\cong \cN\oplus \cN':=\cO(D)\oplus \cO(D').
\end{equation}
Throughout most of the discussion we will also have $\cN\not\cong \cN'$, but this assumption is not necessary to begin with. To fix the notation (specifically, a couple of numerical parameters that will feature prominently below), set
\begin{equation*}
  d:=\deg \cN\le \deg \cN'=: d+k
  \quad\text{with}\quad
  d\ge 2,\ k\ge 0. 
\end{equation*}

The first observation is that \cite[Lemma 3.5]{00-leaves-2} is somewhat sub-optimal: the sections of $\cE$ with torsion-free cokernel are not the only ones generating free $\End(\cE)$-modules.

\begin{lemma}\label{le:morefree}
  Consider a rank-2 bundle as in \Cref{eq:splite}. A section
  \begin{equation*}
    \overline{s}:=(s,s')\in H^0(\cN)\oplus H^0(\cN')\cong H^0(\cE).
  \end{equation*}
  has trivial annihilator in $\End(\cE)$ precisely when $s$ and $s'$ are non-zero and the zero divisor of $s$ is not contained in that of $s'$ (including multiplicities).
\end{lemma}
\begin{proof}
  As the situation is only interesting when both $s$ and $s'$ are non-zero, we assume that throughout. Suppose furthermore, for brevity, that $k>0$ and hence $\deg \cN$ is {\it strictly} smaller than $\cN'$; we leave the other case to the reader as an exercise.

  Our $k>0$ assumption ensures that the embedding $\cN'\le \cE$ is unique up to scaling, so a section $\overline{s}$ with non-zero components $s$ and $s'$ is annihilated by a non-zero endomorphism of $\cE$ precisely when it factors through some embedded copy
  \begin{equation}\label{eq:nsplit}
    \cN\le \cE\cong \cN\oplus \cN'
  \end{equation}
  both of whose components are non-zero; in particular, the left-hand component is a non-zero scalar, since $\End(\cN)\cong \bC$. But then the zero divisor of $\overline{s}\in H^0(\cN)$ (the left-hand side of \Cref{eq:nsplit}) will also be that of $s$, and will be contained in that of $s'$. 

  And conversely, given the zero-divisor containment
  \begin{equation*}
    (s)_0\le (s')_0,
  \end{equation*}
  it follows that $s':\cO\to \cN'$ factors through a morphism $\iota:\cN\to \cN'$. The embedding
  \begin{equation*}
    (\id,\iota):\cN\to \cN\oplus \cN'\cong \cE
  \end{equation*}
  will then have quotient $\cN'$, and composing its cokernel with the embedding $\cN'\le \cE$ gives an endomorphism of $\cE$ that annihilates $\overline{s}$.
\end{proof}

\begin{remark}
  That \Cref{le:morefree} recovers \cite[Lemma 3.5]{00-leaves-2} (in the cases in question) is clear: a section
  \begin{equation*}
    \overline{s}\in H^0(\cE)
  \end{equation*}
  is outside of $\overline{Z}$ precisely when the zero divisors of $s$ and $s'$ do not intersect; of course, this condition is in general stronger than requiring that the larger divisor not {\it contain} the smaller.  
\end{remark}

\begin{definition}\label{def:projgit}
  Let $\cE$ be a rank-2 bundle as in \Cref{eq:splite}.
  \begin{enumerate}[(1)]
  \item A section in $H^0(\cE)$ is {\it non-degenerate} if it has trivial annihilator in $\End(\cE)$; equivalently, if it satisfies the zero-divisor condition of \Cref{le:morefree}.

    The space $H^0(\cE)_{nd}$ of non-degenerate sections is an open subvariety of $H^0(\cE)$, as can be seen by methods analogous to those employed to the same effect on $H^0(\cE)\setminus \overline{Z}$.
  \item The {\it complete homological leaf} $\widetilde{L}(\cE)$ is the geometric quotient
    \begin{equation}\label{eq:biggergit}
      H^0(\cE)_{nd}/\Aut(\cE) \cong \bP H^0(\cE)_{nd}/\bP\Aut(\cE).
    \end{equation}    
  \end{enumerate}
  Again, we will not belabor the matter of geometric-quotient existence: not only do the methods above still apply to the larger space $H^0(\cE)_{nd}$ (larger than $H^0(\cE)\setminus \overline{Z}$, that is), but in fact the quotient can be constructed a good deal more explicitly, as we will discuss below.  
\end{definition}

Returning to (strongly) non-degenerate pairings, consider:

\begin{definition}\label{def:ellbil}
  Let $\cN$ and $\cN'$ be two line bundles on $E$ with $0<\deg \cN\le \deg\cN'$ and set $\cM:=\cN'\otimes \cN^{-1}$. We write
  \begin{equation}\label{eq:pairnn}
    H^0(\cN)\otimes H^0(\cM)
    \xrightarrow{\quad\beta_{\cN,\cN'}\text{ or }\beta_{\cN\prec \cN'}\quad}
    H^0(\cN')
  \end{equation}
  for the obvious section-tensoring map; it is strongly non-degenerate in the sense of \Cref{def:pdpd}.

  The when
  \begin{equation*}
    \cN\cong \cO(D)
    \quad\text{and}\quad
    \cN'\cong \cO(D')
  \end{equation*}
  we use the analogous notation
  \begin{equation*}
    \beta_{D,D'}
    \quad\text{or}\quad
    \beta_{D\prec D'},
  \end{equation*}
  the `$\prec$' symbol being there to remind us of which divisor/bundle has the smaller degree.  
\end{definition}

Finally, to circle back to \Cref{def:projgit}:

\begin{theorem}\label{th:wtlispp}
  Let
  \begin{equation*}
    \cE\cong \cN\oplus \cN',
    \quad
    2\le d:=\deg \cN<\deg \cN'=:d'
  \end{equation*}  
  be a rank-2 bundle as in \cite[\S 2.2]{00-leaves-2}. The complete homological leaf $\widetilde{L}(\cE)$ of \Cref{def:projgit} is isomorphic to the $\bP^{d-1}$-bundle $\bP_{\beta_{\cN\prec \cN'}}$ over
  \begin{equation*}
    \bP^{d-1} \cong \bP H^0(\cN)
  \end{equation*}
  attached to the pairing \Cref{eq:pairnn} as in \Cref{def:pdpd}.
\end{theorem}
\begin{proof}
  We retain the notation
  \begin{equation*}
    \cM:=\cN'\otimes \cN^{-1} 
  \end{equation*}
  of \Cref{def:ellbil}. 

  This is a fairly simple matter of unwinding the definition of $\widetilde{L}(\cE)$: a point therein consists of
  \begin{itemize}
  \item an element $p\in\bP H^0(\cN)$, i.e. a non-zero section modulo scaling by one of the copies of $\bG_m$ in
    \begin{equation*}
      \Aut(\cE)\cong \bG_a^{\deg \cM}\rtimes \bG_m^2;
    \end{equation*}
  \item having fixed $p$, a non-zero class in
    \begin{equation*}
      H^0(\cN')/\left(\text{sections vanishing along the zero divisor $(p)_0$}\right):
    \end{equation*}
    this is the effect of quotienting by the action of $\bG_a^{\deg \cM}$;
  \item and said class is considered only up to scaling: the effect of quotienting out the action of the second copy of $\bG_m$. 
  \end{itemize}
  In summary: $\widetilde{L}(\cE)$ is the projectivization of the vector bundle over $\bP H^0(\cN)$ whose fiber over $p\in \bP H^0(\cN)$ is
  \begin{equation*}
    H^0(\cN')/\left(\text{image of }\beta_{\cN\prec \cN'}(p,-)\right)
  \end{equation*}
  (slightly abusing notation: the first argument of $\beta$ is a vector, but the image only depends on the line containing that vector). This, however, is precisely the definition of $\bP_{\beta_{\cN\prec \cN'}}$.  
\end{proof}

In view of \Cref{th:wtlispp}, it will be of some interest to better understand the ``elliptic pairings'' \Cref{eq:pairnn}. First, we fix some notation that will recur frequently (as it already has).

\begin{notation}\label{not:ddnn}
  We will typically work with two divisors on the elliptic curve $E$ and their associated line bundles:
  \begin{align*}
    \cN=\cO_E(D),&\quad d:=\deg \cN\ge 2\\
    \cN'=\cO_E(D'),&\quad d':=\deg \cN > d\\
    \cM:=\cN'\otimes \cN^{-1},&\quad k:=\deg \cM = d'-d\ge 1.
  \end{align*}  
  
\end{notation}

\begin{lemma}\label{le:whensurj}
  For line bundles as in \Cref{not:ddnn} the section-multiplication map
  \begin{equation*}
    \beta:=\beta_{\cN\prec \cN'}:H^0(\cN)\otimes H^0(\cM)\to H^0(\cN')
  \end{equation*}
  of \Cref{eq:pairnn} is onto if and only if
  \begin{itemize}
  \item $k=\deg \cM\ge 2$;
  \item and, in case $d=k=2$, $\cN\not\cong \cM$. 
  \end{itemize}  
\end{lemma}
\begin{proof}
  
  
  Surjectivity in the listed cases follows from \cite[Theorem 2]{mum-q} with (in that statement's notation) $\cF:=\cN$ and $L:=\cM$.

  Conversely, when $\cN\cong \cM$ are both of degree 2 $\beta$ of course factors through the 3-dimensional symmetric square $S^2 H^0(\cN)$, so cannot have 4-dimensional image $H^0(\cN')$. As for the case $k=1$, we then have $\cM\cong \cO(x)$ for some $x\in E$ and hence all elements in the image of $\beta$ vanish at $x$.
\end{proof}

\subsection{Arbitrary $d$: anticanonical complements}\label{subse:lgd}

We work with decomposable rank-2 bundles $\cE$ as in \Cref{eq:splite} with \Cref{not:ddnn} still in force, $d=\deg D\ge 2$ is now arbitrary. We also set
\begin{equation*}
  V:=H^0(\cN),\ V':=H^0(\cN'),\ W:=H^0(\cM),
\end{equation*}
and the strongly non-degenerate pairing 
\begin{equation*}
  \beta:V\otimes W\to V'
\end{equation*}
is the $\beta_{\cN\prec \cN'}$ of \Cref{eq:pairnn}.

The focus of the discussion is the divisor
\begin{equation}\label{eq:divc}
  Y=Y_{\cE}:=\widetilde{L}(\cE)\setminus L(\cE),
\end{equation}
image of (an open subset of) $\bP Z$ under the geometric quotient \Cref{eq:biggergit}. We need a number of auxiliary objects, together with the ancillary notation.

\begin{definition}
  Let $D$ be a divisor of degree $d\ge 2$ on the elliptic curve $E$ and set $\cN:=\cO(D)$, We denote by $E^{(D)}$ or $E^{(\cN)}$ the $(d-1)$-dimensional variety
  \begin{equation*}
    \{(z,(z_1,\ \cdots,\ z_{d-1}))\in E\times E^{(d-1)}\ |\ z+\sum z_i = \sigma(D)\},
  \end{equation*}
  where $E^{(d-1)}$ is the $(d-1)^{st}$ symmetric power of $E$.

  The projection
  \begin{equation}\label{eq:ede}
    \pi=\pi_{D}=\pi_{\cN}:E^{(D)}\to E
  \end{equation}
  onto the first component is a $\bP^{d-2}$-fibration; it can be identified with the usual {\it Abel-Jacobi map} $E^{(d-1)}\to E$ (\cite[Chapter I, equation (3.2)]{acgh1}, \cite[\S 1]{CaCi93}) that sends an effective degree-$d$ divisor to the sum over its support: pull back the Abel-Jacobi fibration by the automorphism
  \begin{equation*}
    E\ni z\xmapsto{\quad}\sigma(D)-z\in E
  \end{equation*}
  of $E$.
\end{definition}

Apart from \Cref{eq:ede}, another morphism that will feature below is 
\begin{equation}\label{eq:edp}
  E^{(D)}\ni (z,(z_1,\ \cdots,\ z_{d-1}))
  \xmapsto{\quad q\quad}
  (z,\ z_1,\ \cdots,\ z_{d-1})
  \in \bP^{d-1}\cong \bP H^0(\cN)
\end{equation}
(having identified the line through a non-zero section of $\cN$ with its divisor of zeros).

The two allow (i.e. \Cref{eq:ede} and \Cref{eq:edp}) us to pull back a number of bundles on to $E^{(D)}$:
\begin{itemize}
\item We have, on the one hand, an inclusion
  \begin{equation}\label{eq:qsb}
    q^* \cS_{\beta}\subset q^*(\cO_{\bP^{d-1}}\otimes V')\cong \cO_{E^{(D)}}\otimes H^0(\cN'),
  \end{equation}
  as the pullback along \Cref{eq:edp} of the inclusion $\cS_{\beta}\subset \cO\otimes V'$ of bundles over $\bP^{d-1}\cong \bP H^0(\cN)$ (\Cref{def:pdpd}).
\item On the other hand, the inclusions
  \begin{equation*}
    H^0(\cN'(-z))\subset H^0(\cN')
  \end{equation*}
  are the fibers of an $E$-vector-bundle embedding
  \begin{equation*}
    \cS\subset \cO_E\otimes H^0(\cN').
  \end{equation*}
  That inclusion can then be pulled back along \Cref{eq:ede} to
  \begin{equation}\label{eq:pis}
    \pi^* \cS\subset \cO_{E^{(D)}}\otimes H^0(\cN'). 
  \end{equation}
\end{itemize}

The various objects fit together neatly:

\begin{lemma}\label{le:flag3}
  The inclusions \Cref{eq:qsb} and \Cref{eq:pis} fit into a flag
  \begin{equation*}
    q^* \cS_{\beta}
    \subset
    \pi^* \cS
    \subset
    \cO_{E^{(D)}}\otimes H^0(\cN')
  \end{equation*}
  of vector bundles over $E^{(D)}$, of respective ranks $d'-d$, $d'-1$ and $d'$.
\end{lemma}
\begin{proof}
  This is a simple matter of unwinding the definitions: at
  \begin{equation}\label{eq:pted}
    (z,(z_1,\ \cdots,\ z_{d-1})) \in E^{(D)}
  \end{equation}
  the fiber of $q^*\cS_{\beta}$ consists of those sections of $\cN'$ that vanish at $z$ and all $z_i$, while that of $\pi^*\cS$ consists of those that vanish at $z$. Naturally, the former condition is more restrictive, hence the first claimed inclusion (the other being obvious). 

  As for the ranks, they are the degrees of $\cM=\cN'\otimes \cN^{-1}$, $\cN'(-z)$ (for any $z\in E$) and $\cN'$.
\end{proof}

According to \Cref{le:flag3}, $\pi^*\cS/q^*\cS_{\beta}$ is a vector bundle of rank $d-1$ over $E^{(D)}$. Furthermore, $q$ induces a morphism (denoted abusively by the same symbol)
\begin{equation*}
  q: \bP (\cO\otimes V'/q^*\cS_{\beta})\to \bP (\cO\otimes V'/q^*\cS_{\beta}) = \bP_{\beta}
\end{equation*}
between equidimensional projectivizations, which then restricts to
\begin{equation}\label{eq:qres}
  q_{res}:\bP(\pi^*\cS/q^*\cS_{\beta}) \to \bP_{\beta};
\end{equation}
note that the domain and codomain of $q_{res}$ have dimensions $2d-3$ and $2d-2$ respectively. This map will give a handle of sorts on the divisor $Y=Y_{\cE}$ of $\bP_{\beta}$, as we show in \Cref{th:ynorm}.

That statement requires a common piece of terminology: an {\it anticanonical divisor} is one that is inverse to the canonical divisor of \cite[Example V.1.4.4]{hrt}, up to linear equivalence. The concept provides a compact phrase for identifying the difference between the homological leaf and its ambient completion:

\begin{theorem}\label{th:ynorm}
  \begin{enumerate}[(a)]
  \item\label{item:ynorm-imy} The image of \Cref{eq:qres} is precisely the divisor $Y$ of \Cref{eq:divc}, which is anticanonical in $\bP_{\beta}$.
  \item\label{item:ynorm-norm} Furthermore, $q_{res}$ identifies its domain $\bP(\pi^*\cS/q^*\cS_{\beta})$ with the normalization of its image $Y$. 
  \item\label{item:ynorm-cores} The corestriction
    \begin{equation}\label{eq:qrescores}
      q_{co}:\bP(\pi^*\cS/q^*\cS_{\beta}) \to Y
    \end{equation}
    is finite, with maximal fiber length $\ge d-1$.
  \item\label{item:ynorm-iso2} In particular, \Cref{eq:qrescores} is an isomorphism if and only if $d=2$.
  \end{enumerate}
\end{theorem}
\begin{proof}
  Throughout the proof we write
  \begin{equation*}
    Y'=Y'_{\cE}:= \bP(\pi^*\cS/q^*\cS_{\beta})
  \end{equation*}
  for the domain of \Cref{eq:qrescores}. 

  \vspace{.5cm}
  
  {\bf \Cref{item:ynorm-imy}} Consider a point
  \begin{equation}\label{eq:zzd}
    (z_1,\ \cdots,\ z_{d})\in \bP^{d-1}=\bP H^0(\cN)
  \end{equation}
  That point cuts out a $d$-codimensional (i.e. $(d'-d)$-dimensional) subspace
  \begin{equation}\label{eq:hzs}
    H^0(\cN')_{\left(z_i,\ 1\le i\le d\right)}\subset H^0(\cN')
  \end{equation}
  consisting of those sections that vanish at the $z_i$ (counting multiplicities, i.e. vanish to the appropriate higher order if there are duplicates among the $z_i$). Finally, the fiber of $\bP_{\beta}$ above \Cref{eq:zzd} consists of the lines in
  \begin{equation*}
    H^0(\cN')/H^0(\cN')_{\left(z_i,\ 1\le i\le n\right)}
  \end{equation*}
  or, equivalently, of the $(d-1)$-codimensional subspaces of $H^0(\cN')$ containing \Cref{eq:hzs}. 

  Such a subspace belongs to the divisor $Y$ precisely when it consists of sections vanishing at {\it some} $z_i$, $1\le i\le d$. To conclude, note that
  \begin{itemize}
  \item The fiber of $Y'$ above \Cref{eq:pted} consists of the $(d-1)$-codimensional subspaces of $H^0(\cN')$ containing \Cref{eq:hzs} and contained in $H^0(\cN')_{z}$;
  \item and the morphism \Cref{eq:qres} operates as \Cref{eq:edp} on the base and identifies collections of $(d-1)$-dimensional subspaces of $H^0(\cN')$ in the obvious fashion.
  \end{itemize}
  As for the claim that $Y$ is anticanonical, it requires a lengthier discussion that we defer for the moment; the claim itself is \Cref{pr:anticfinal}.

  
  \vspace{.5cm}
  
  {\bf \Cref{item:ynorm-norm} and \Cref{item:ynorm-cores} } We treat these as a unit because both the finiteness claim and the fact that \Cref{eq:qrescores} is a normalization follow if we prove that that morphism (or equivalently, \Cref{eq:qres}) is only {\it quasi-}finite (i.e. has finite fibers: \cite[D\'efinition 6.2.3]{ega2}) and {\it birational} (i.e. induces an isomorphism between dense open subsets of its domain and codomain \cite[\S 2.2.9]{ega1}), as we now explain.

  \vspace{.2cm}
  
  {\bf Assume $q_{co}$ is quasi-finite and birational.} It is then also finite by \cite[Corollaire 4.4.11]{ega31}, being a morphism between projective schemes and hence projective. But then it is in particular also {\it integral} \cite[D\'efinition 6.1.1]{ega2}, and the conclusion follows from \cite[Lemme 17.15.14.1]{ega44}.

  \vspace{.2cm}

  {\bf $q_{res}$ (and hence $q_{co}$) is quasi-finite, with maximal fiber length $\le d$.} By its very definition, \Cref{eq:qres} is a morphism of bundles over $E^{(D)}$ and $\bP^{d-1}=\bP H^0(\cN)$, in the sense that it fits into the commutative diagram
  \begin{equation*}
    \begin{tikzpicture}[auto,baseline=(current  bounding  box.center)]
      \path[anchor=base] 
      (0,0) node (lu) {$Y'$}
      +(2,0) node (ru) {$\bP_{\beta}$}
      +(0,-1) node (ld) {$E^{(D)}$}
      +(2,-1) node (rd) {$\bP^{d-1}$}
      ;
      \draw[->] (lu) to[bend left=0] node[pos=.5,auto] {$\scriptstyle q_{res}$} (ru);
      \draw[->] (ld) to[bend left=0] node[pos=.5,auto,swap] {$\scriptstyle q$} (rd);
      \draw[->] (lu) to[bend left=0] node[pos=.5,auto] {$\scriptstyle $} (ld);
      \draw[->] (ru) to[bend left=0] node[pos=.5,auto] {$\scriptstyle $} (rd);
    \end{tikzpicture}
  \end{equation*}
  The conclusion follows from the fact that $q_{res}$ restricts to linear embeddings on the fibers, and to the degree-$d$ surjection \Cref{eq:edp} on the base. 

  \vspace{.2cm}

  {\bf For some dense open subset $U\subseteq Y$, the restriction
    \begin{equation*}
      q_{co}:q_{co}^{-1}(U)\to U
    \end{equation*}
    is \'etale. 
  }

  The restriction of $q_{co}$ to some open dense $U'\subseteq Y'$ is \'etale because we are working over an algebraically closed field of characteristic zero \cite[Lemma III.10.5]{hrt}. We also know that $q_{co}$ is finite (once more by \cite[Corollaire 4.4.11]{ega31}) and hence closed, because it is quasi-finite and projective. Now take
  \begin{equation*}
    U:=Y - q_{co}(Y'\setminus U').
  \end{equation*}

  \vspace{.2cm}
  
  {\bf The set-theoretic fibers of $q_{co}$ over some dense open $U\subseteq Y$ are singletons.} First, we restrict attention only to (the elements of $Y$ lying on) the fibers $\bP_{\beta,{\bf z}}$ of $\bP_{\beta}\to \bP^{d-1}$ above elements
  \begin{equation*}
    {\bf z} = (z_1,\ \cdots,\ z_{d})\in \bP^{d-1}=\bP H^0(\cN)
  \end{equation*}  
  with {\it distinct} components $z_i$. Since the collection of such ${\bf z}$ is open and non-empty (hence dense) in $\bP^{d-1}$, so is
  \begin{equation*}
    \bigcup_{\bf z}Y_{\bf z}\subset Y,
    \quad Y_{\bf z}:=Y\cap \bP_{\beta,{\bf z}}.
  \end{equation*}

  Next, for such a point ${\bf z}$, temporarily relabeled
  \begin{equation}\label{eq:relabel}
    {\bf z} = (z,\ z_1,\ \cdots,\ z_{d-1})
  \end{equation}
  for convenience, consider only those $(d-1)$-codimensional subspaces of $H^0(\cN')_{z}$ of the form
  \begin{equation}\label{eq:hzcs}
    H^0(\cN')_{{\bf z}}+\bC s,\ s\in H^0(\cN')_{z}\text{ vanishing at no $z_i$},1\le i\le d-1. 
  \end{equation}
  These are elements of $Y_{\bf z}$. They are by construction contained in the image through \Cref{eq:qres} of the fiber
  \begin{equation*}
    Y'_{(z,\ (z_1,\ \cdots,\ z_{d-1}))}\subset Y'
  \end{equation*}
  above
  \begin{equation*}
    (z,\ (z_1,\ \cdots,\ z_{d-1}))\in E^{(D)},
  \end{equation*}
  and belong to the $q_{res}$-image of no {\it other} fiber of $Y'\to E^{(D)}$. The desired open dense set $U\subset Y$ then consists of all such \Cref{eq:hzcs}, for
  \begin{itemize}
  \item points ${\bf z}\in \bP^{d-1}$ with distinct components;
  \item and relabelings \Cref{eq:relabel} singling out each of the $d$ components in turn;
  \item and $s\in H^0(\cN')_{z}$ as in \Cref{eq:hzcs}.
  \end{itemize}
  
  \vspace{.2cm}
  
  {\bf $q_{co}$ is birational.} The two preceding claims provide us with an open dense $U\subseteq Y$ (e.g. the intersection of the two opens respectively discussed there) such that
  \begin{equation}\label{eq:qcou}
    q_{co}:q_{co}^{-1}(U)\to U
  \end{equation}
  is \'etale and has singleton (reduced, by \'etale-ness) fibers. \Cref{eq:qcou} is also finite, being a base-change of the original finite morphism $q_{co}$. The claim now follows, say, from \spr{04DH}.

  Per the discussion so far, the only claim left to prove in order to complete parts \Cref{item:ynorm-norm} and \Cref{item:ynorm-cores} is

  \vspace{.2cm}

  {\bf $q_{co}$ has at least one set-theoretic fiber of cardinality $\ge d-1$.} Consider, first, a point \Cref{eq:zzd} with distinct components $z_i$. The space
  \begin{equation*}
    H^0(\cN')_{\left(z_i,\ 1\le i\le d\right)}
  \end{equation*}
  of sections vanishing at all $z_i$ is a hyperplane in, say,
  \begin{equation}\label{eq:hzdsmall}
    H^0(\cN')_{\left(z_i,\ 1\le i\le d-1\right)}.
  \end{equation}
  The quotient of the two is an element of $\bP_{\beta}$, in the fiber above \Cref{eq:zzd}, and is in the image through $q_{co}$ of elements of $Y'$, one in each of the $d-1$ fibers above
  \begin{equation*}
    (z_i,(z_1,\ \cdots,\ z_{i-1},\ z_{i+1},\ \cdots,\ z_{d})) \in E^{(D)}
  \end{equation*}
  for $1\le i\le d-1$ (i.e. those $i$ appearing in \Cref{eq:hzdsmall}).
  

  \vspace{.5cm}
  
  
  {\bf \Cref{item:ynorm-iso2}} For $d=2$ $Y'$ is nothing but $E$, and the morphism \Cref{eq:qres} is a closed immersion. Indeed, that morphism sends $z\in E$ to the hyperplane
  \begin{equation*}
    H^0(\cN')_{z}\subset H^0(\cN'),
  \end{equation*}
  regarded as an element of the fiber of $\bP_{\beta}$ above $(z,\sigma(D)-z)\in \bP^1\cong \bP H^0(\cN)$. Since $\cN'$ has degree $\ge 3$ and is thus very ample \cite[Corollary IV.3.2]{hrt},
  \begin{equation*}
    E\ni z\mapsto H^0(\cN')_{z}\in \bP H^0(\cN')^*
  \end{equation*}
  is already a closed immersion.

  On the other hand, according to \Cref{item:ynorm-cores}, for $d>2$ the morphism \Cref{eq:qrescores} is not one-to-one.
\end{proof}

\subsection{The case $d=2$: Hirzebruch surfaces}\label{subse:d2}

We again fix bundles $\cN$ and $\cN'$ on $E$ as in \Cref{not:ddnn} (with the attendant lettering) and
\begin{equation*}
  \cE\cong \cN\oplus \cN'.
\end{equation*}
We specialize the preceding discussion to the case $d=2$, once more with the goal of understanding the ambient space $\widetilde{L}(\cE)$ of the homological leaf $L(\cE)$ and how the latter sits therein in the particular case when $d=2$.

The main result, aggregating the entire view of the situation, is spread over \Cref{pr:homishirz,th:isantican,th:d2hirz}. We will use some of the language and notation familiar from the theory of ruled surfaces:
\begin{itemize}
\item The invariant $e$ of the Hirzebruch surface $\Sigma_e$ (mentioned also in \Cref{ex:hirz}) is the same as that appearing throughout \cite[\S V.2]{hrt}.
\item We make extensive use of the {\it intersection pairing}
  \begin{equation*}
    (C,C')\mapsto C.C'
  \end{equation*}
  of \cite[\S V.1]{hrt} for divisors $C$ and $C'$ on a surface (for us, one of the Hirzebruch surfaces $\Sigma_e$).
\item For a Hirzebruch surface $\Sigma_e$ fibering over $\bP^1$ as in \Cref{ex:hirz} (with $e\in \bZ_{\ge 0}$) we follow \cite[Notation V.2.8.1]{hrt} in denoting by $C_0$ (the image of) a section $\bP^1\to \Sigma_e$ with self-intersection
  \begin{equation*}
    C_0^2:=C_0.C_0=-e.
  \end{equation*}
\item Per the same source, $f$ denotes (the class of) a fiber of $\Sigma_e\to \bP^1$. By \cite[Proposition V.2.3]{hrt}, the Picard and N\'eron-Severi groups of $\Sigma_e$ are both isomorphic to $\bZ^2$, with $C_0$ and $f$ serving as a basis.
\item The intersection pairing is then uniquely determined by
  \begin{equation}\label{eq:intrels}
    C_0^2 = -e,\quad C_0.f=1,\quad f^2=0
  \end{equation}
  \cite[Proposition V.2.3]{hrt}.
\end{itemize}

We begin with a relatively general remark, that by now is a simple consequence of the preceding material.

\begin{proposition}\label{pr:homishirz}
  Let $\cN=\cO(D)$ and $\cN'=\cO(D')$ be line bundles on $E$ as in \Cref{not:ddnn} with $d=\deg\cN=2$ and set $\cE=\cN\oplus \cN'$.

  The complete homological leaf $\widetilde{L}(\cE)$ is isomorphic to a Hirzebruch surface $\Sigma_e$, for some
    \begin{equation*}
      0\le e\le k=d'-d=d'-2
    \end{equation*}
    of the same parity as $k$ (or $d'$, or the total degree $n:=\deg\cE=d'+2=k+4$).
\end{proposition}
\begin{proof}
  We first apply \Cref{th:wtlispp}. Because $d=2$, we are within the scope of \Cref{le:allnonneg} with $\dim V'-\dim W=2$ summands:
  \begin{equation*}
    \widetilde{L}(\cE)\cong \bP_{\beta_{\cN\prec \cN'}}\cong \bP(\cO(d_1)\oplus \cO(d_2))
  \end{equation*}
  for some $d_1\ge d_2\ge 0$ with $d_1+d_2=\dim W=k$. Bundle projectivizations are invariant under tensoring the original vector bundle by line bundles \cite[Exercise II.7.9]{hrt}, so that
  \begin{equation*}
    \widetilde{L}(\cE)\cong \bP(\cO(d_1)\oplus \cO(d_2))\cong \bP(\cO(d_1-d_2)\oplus \cO). 
  \end{equation*}
  But this is precisely $\Sigma_{d_1-d_2}$, and the conclusion follows by setting $e:=d_1-d_2$ (which notation will be in place throughout the proof); naturally, given
  \begin{equation*}
    d_1\ge d_2\ge 0,\ d_1+d_2=k,
  \end{equation*}
  the possibilities are precisely as claimed: ranging between $0$ and $k$, and having the same parity as $k$.
\end{proof}

In the present setting, \Cref{th:ynorm} specializes to the following statement; having postponed the {\it general} proof that $Y$ is anticanonical, we provide an alternative, ad-hoc argument verifying the claim in the particular case $d=2$.

\begin{theorem}\label{th:isantican}
  Let $\cN=\cO(D)$ and $\cN'=\cO(D')$ be line bundles on $E$ as in \Cref{not:ddnn} with $d=\deg\cN=2$ and set $\cE=\cN\oplus \cN'$.

  The divisor \Cref{eq:divc}
  \begin{equation*}
    Y=Y_{\cE}:=\widetilde{L}(\cE)\setminus L(\cE)
  \end{equation*}
  is the image of a closed immersion $E\subset \Sigma_e$ and anticanonical.
\end{theorem}
\begin{proof}
We now know that
  \begin{itemize}
  \item $Y$ is an elliptic curve and hence has genus $g_Y=1$;
  \item and is of class $2C_0+tf$ in the Picard group of $\Sigma$.
  \end{itemize}
  Denoting by $K$ the canonical divisor of $\Sigma$, the {\it adjunction formula} of \cite[Proposition V.1.5]{hrt} reads
  \begin{equation}\label{eq:adjform}
    0=2g_Y-2 = Y.(Y+K). 
  \end{equation}
  On $\Sigma\cong \Sigma_e$ we also know that
  \begin{equation*}
    K=-2C_0-(2+e)f
  \end{equation*}
  by \cite[Corollary V.2.11]{hrt}, so the coefficient of $C_0$ in the right-hand factor $Y+K$ of \Cref{eq:adjform} vanishes. That formula then simplifies to
  \begin{equation*}
    0=(2C_0+tf).(t-2-e)f=0\Longrightarrow t-2-e=0,
  \end{equation*}
  so that in fact $Y=-K$. Or: $Y$, as claimed, is anticanonical.  
\end{proof}

\begin{theorem}\label{th:d2hirz}
  Let $\cN=\cO(D)$ and $\cN'=\cO(D')$ be line bundles on $E$ as in \Cref{not:ddnn} with $d=\deg\cN=2$ and set $\cE=\cN\oplus \cN'$.
  \begin{enumerate}[(a)]
  \item\label{item:012} The only possible values for $e$ are 0, 1 and 2.
  \item\label{item:nodd1} In particular, when $n$ is odd we have $\widetilde{L}(\cE)\cong \Sigma_1$, $Y$ is ample, and hence its complement
    \begin{equation*}
      L(\cE) = \widetilde{L}(\cE)\setminus Y
    \end{equation*}
    is affine.
  \item\label{item:neven02} On the other hand, when $n$ (and hence $k=n-4$) is even we have    
    \begin{equation*}
      \widetilde{L}(\cE)\cong
      \begin{cases}
        \Sigma_2 &\text{if $D'$ is a multiple of $D$, i.e. $D'\sim \left(\frac k2+1\right)D$};\\
        \Sigma_0 &\text{otherwise}.                   
      \end{cases}
    \end{equation*}
  \item\label{item:whenaffev} For even $n$ the following conditions are equivalent:
    \begin{enumerate}[(i)]
    \item $D'$ is not a multiple of $D$;
    \item $Y$ is (very) ample;
    \item $L(\cE)$ is affine;
    \item $L(\cE)$ is quasi-affine.    
    \end{enumerate}
  \end{enumerate}
\end{theorem}
\begin{proof}  
  {\bf \Cref{item:012}} We established in \Cref{th:isantican} that the anticanonical class $2C_0+(e+2)f$ contains an irreducible curve (namely the elliptic curve $C$). According to \cite[Corollary V.2.18 (b)]{hrt} this is possible only when
  \begin{equation}\label{eq:eineq}
    e+2\ge 2e\Longleftrightarrow 2\ge e.
  \end{equation}

  \vspace{.5cm}
  
  {\bf \Cref{item:nodd1}} Part \Cref{item:012} and the parity portion claim in \Cref{pr:homishirz} indeed, for odd $n$ the only possibility is $e=1$. In that case the inequality \Cref{eq:eineq} is in fact strict, and the anticanonical divisor is (very) ample by \cite[Corollary V.2.18 (a)]{hrt}. Its complement is then affine, as noted in the footnote in the proof of \cite[Corollary 3.17]{00-leaves-2}.
  
  \vspace{.5cm}
  
  {\bf \Cref{item:neven02}} We know from \Cref{pr:homishirz} and part \Cref{item:012} that the only possibilities for even $k=:2\ell$ are $e=0$ and $e=2$, so it remains to identify various distinguishing characteristics between the two cases.

  \begin{enumerate}[(i)]
  \item Suppose first that $D'$ {\it is} a multiple of $D$:
    \begin{equation}\label{eq:dismult}
      D'\sim \left(\frac k2+1\right)D = (\ell+1)D. 
    \end{equation}
    The goal is to show that $e=2$, i.e. our ruled surface is $\Sigma_2$ rather than $\Sigma_0$.

    One feature that distinguishes the two is the fact that in $\Sigma_2$ the canonical section $C_0$ has trivial intersection with the anticanonical divisor $2C_0+4 f$ (and hence with the divisor $C$ of \Cref{eq:divc}), whereas in $\Sigma_0$ this is not the case: the anticanonical divisor $2C_0+2f$ of $\Sigma_0$ is very ample by \cite[Corollary V.2.18]{hrt} so certainly, it will intersect every effective divisor.

    It will thus be enough to produce a section of
    \begin{equation}\label{eq:projtop1}
      \widetilde{L}(\cE)\cong \Sigma_e\xrightarrow{\quad\pi\quad} \bP^1\cong \bP H^0(\cN)=\bP H^0(\cO(D)).
    \end{equation}
    that fails to intersect $C$. Using the notation in the proof of \Cref{th:isantican}, recall the description of the fiber $\pi^{-1}(\{x,x'\})$ as consisting of intermediate hyperplanes
    \begin{equation*}
      H^0(\cO(D'))_{x,x'}\subset \bullet\subset H^0(\cO(D')).
    \end{equation*}
    A section of $\pi$, then, is a selection of a unique such hyperplane above each $\{x,x'\}$ (continuous, etc.).
    
    Under the present assumption \Cref{eq:dismult} we construct a distinguished section as follows. We have an isomorphism
    \begin{equation*}
      \cN' = \cO(D') = \cO((\ell+1)D) = \cN^{\otimes(\ell+1)},
    \end{equation*}
    and hence a corresponding morphism
    \begin{equation}\label{eq:psil1}
      S^{\ell+1}H^0(\cN)\xrightarrow{\quad\psi\quad} H^0(\cN^{\otimes(\ell+1)}) \cong H^0(\cN')
    \end{equation}
    that is easily seen to be an embedding; the dimension of the domain (so also that of the image) is $\ell+2$. For each $\{x,x'\}\in \bP H^0(\cN)$ we have
    \begin{equation*}
      \dim\left(H^0(\cN')_{x,x'}\cap \mathrm{Im}~\psi\right)=\ell+1 = \dim \left(\mathrm{Im}~\psi\right)-1.
    \end{equation*}
    It follows that
    \begin{equation*}
      H^0(\cN')_{x,x'}+\mathrm{Im}~\psi\le H^0(\cN')
    \end{equation*}
    has dimension $k+1$, thus containing the $k$-dimensional $\dim H^0(\cN')_{x,x'}$ as a hyperplane. This, then, is our section to \Cref{eq:projtop1}:
    \begin{equation}\label{eq:specsect}
      \bP H^0(\cN)\ni \{x,x'\}\xmapsto{\quad} H^0(\cN')_{x,x'}+\mathrm{Im}~\psi.
    \end{equation}
    That it intersects $C$ trivially is easy to see: the intersection of $C$ with the fiber above $\{x,x'\}$ consists of the hyperplanes
    \begin{equation*}
      H^0(\cN')_x\quad\text{and}\quad H^0(\cN')_{x'}
    \end{equation*}
    of sections vanishing at $x$ and $x'$ (perhaps a single hyperplane counted with multiplicity if $x=x'$), and it is plain that no $H^0(\cN')_x$ contains $\mathrm{Im}~\psi$, and hence cannot coincide with any of the hyperplanes on the right-hand side of \Cref{eq:specsect}.       
    
  \item We now assume
    \begin{equation*}
      D'\not\sim (\ell+1)D,
    \end{equation*}
    and seek to show that $e=0$. Again, the argument hinges on constructing sections to \Cref{eq:projtop1} with certain desired properties.

    Specifically, we will exhibit a section whose intersection number with the anticanonical divisor $Y$ is 2. This will suffice, as such sections do not exist on $\Sigma_2$. Indeed, note first that being a section the divisor in question would have to be, numerically, of the form
    \begin{equation*}
      C_0+tf,\ t\in \bZ.
    \end{equation*}
    The equations
    \begin{equation*}
      (C_0+tf).(2C_0+4f)=2 
    \end{equation*}
    and \Cref{eq:intrels} (with $e=2$, which we are assuming for contradiction's sake) then force $t=1$, but \cite[Corollary V.2.18 (b)]{hrt} shows that on $\Sigma_2$ there are no {\it irreducible} divisors of class $C_0+f$.

    To construct the desired sections we proceed much as above, thus leveraging what is hopefully a partly familiar pattern. First write
    \begin{equation*}
      D'\sim \ell D+D'',\ D''\not\sim D,
    \end{equation*}
    and fix a non-zero section $s\in H^0(D'')$ vanishing, say, at $\{y,y'\}$.

    Tensoring $s:\cO\to \cO(D'')$ with $\cO(D'-D'')$ we obtain an embedding
    \begin{equation*}
      \cO(D'-D'')\cong \cO(D)^{\otimes \ell}\cong \cN^{\otimes \ell}\to\cN'\cong \cO(D'),
    \end{equation*}
    which we can further compose with the analogue of \Cref{eq:psil1} ($\ell$-fold rather than $\ell+1$) to obtain a (mono)morphism
    \begin{equation*}
      S^{\ell}H^0(\cN)\xrightarrow{\quad\psi\quad} H^0(\cN').
    \end{equation*}
    Its image now has dimension $\ell+1$ (one less than before, in the preceding version of the argument), intersecting it with any
    \begin{equation*}
      H^0(\cN')_{x,x'},\ \{x,x'\}\in \bP H^0(\cN)
    \end{equation*}
    still brings down that dimension by 1:
    \begin{equation*}
      \dim\left(H^0(\cN')_{x,x'}\cap \mathrm{Im}~\psi\right)=\ell = \dim \left(\mathrm{Im}~\psi\right)-1
    \end{equation*}
    because
    \begin{equation*}
      D\not\sim D''\Longrightarrow \sigma(D)=x+x'\ne y+y'=\sigma(D'')\in E.
    \end{equation*}

    We can now once more construct a section \Cref{eq:specsect} of \Cref{eq:projtop1}, which this time around intersects $C$ twice: assuming $y\ne y'$ (which we can always arrange by selecting $s$ generically), the $\{x,x'\}\in \bP H^0(\cN)$ for which 
    \begin{equation*}
      H^0(\cN')_{x,x'}+\mathrm{Im}~\psi = H^0(\cN')_z\text{ for some }z\in E
    \end{equation*}
    are precisely the two with
    \begin{equation*}
      y\text{ or }y'\in \{x,x'\},
    \end{equation*}
    in which case $z=y$ or $y'$ respectively.
  \end{enumerate}
  This concludes the proof of part \Cref{item:neven02}.
  
  \vspace{.5cm}
  
  {\bf \Cref{item:whenaffev}} We first prove the downward implications:
  \begin{itemize}
  \item If $D'$ is not a multiple of $D$ then by part \Cref{item:neven02} we have $e=0$. But then the inequality \Cref{eq:eineq} is again strict, so that $Y$ is very ample by the already-cited \cite[Corollary V.2.18 (a)]{hrt}.
  \item If $Y$ is ample then its complement is affine, as observed/recalled repeatedly.
  \item Naturally, affineness implies quasi-affineness.
  \end{itemize}
  Finally, for the contrapositive of the bottom-to-top implication, suppose $D'$ is {\it not} a multiple of $D$. We then have $e=2$ by part \Cref{item:neven02}, so that \Cref{eq:eineq} is an equality. We then have
  \begin{equation*}
    C.C_0 = (2C_0+4f).C_0 = -4+4=0,
  \end{equation*}
  so that $C$ fails to intersect $C_0$. But then
  \begin{equation*}
    L(\cE) = \widetilde{L}(\cE)\setminus Y
  \end{equation*}
  contains a projective curve $C_0\cong \bP^1$, so cannot be quasi-affine.
  
  This finishes the proof of the theorem as a whole. 
\end{proof}

\begin{remark}\label{re:uniqsect}
  In the first part of the proof of \Cref{th:d2hirz} \Cref{item:neven02}, where we sought to show that $\widetilde{L}(\cE)$ is $\Sigma_2$, the section \Cref{eq:specsect} was unique with its requisite properties: for $e>0$ the fibration $\Sigma_e\to \bP^1$ has a unique section with self-intersection $-e$ \cite[Example V.2.11.3]{hrt}. Exhibiting \Cref{eq:specsect} was thus a matter of finding the {\it unique} such object, and the construction had a certain inevitability to it.

  On the other hand, in the second part of that proof, where the goal was to prove
  \begin{equation*}
    \widetilde{L}(\cE)\cong \Sigma_0\cong \bP^1\times \bP^1,
  \end{equation*}
  we made a choice of a non-zero section of the degree-2 line bundle $\cO(D'')$. This reflects the fact that there is a $\bP^1$ family of sections to the (second, say) projection $\bP^1\times\bP^1\to \bP^1$ with zero self-intersection.  
\end{remark}

\section{Codimension-3 smoothness}\label{se:codim3}

Suppose $E\subset \bP^{n-1}$ is a normal elliptic curve (\cref{foot:norm}), and fix a positive integer $d<\frac n2$. It is a consequence of \cite[Theorem]{copp} that the singular locus of the $d$-secant variety
\begin{equation*}
  \Sec_d=\Sec_d(E)\subset \bP^{n-1}
\end{equation*}
is precisely the lower secant variety $\Sec_{d-1}$.

\cite[p.18, Proof, part 1]{copp} addresses one implication of that claim: $\Sec_{d-1}$ consists entirely of singular points of $\Sec_d$. That argument proves more: the tangent space at any point
\begin{equation*}
  x\in \Sec_{d-1}\subset \Sec_d
\end{equation*}
is full, i.e. $(n-1)$-dimensional.

We focus here instead on the slices $\Sec_{d,z}(E)$ of the secant variety, consisting of the points in $\bP^{n-1}$ on $(d-1)$-planes through degree-$d$ divisors with sum $z\in E$:
\begin{equation}\label{eq:secdz}
  \Sec_{d,z}=\Sec_{d,z}(E):=\bigcup_{\deg D=d,\ \sigma(D)=z} \overline{D},
\end{equation}
where
\begin{itemize}
\item $\sigma(-)$ denotes summation over the support of an effective divisor (counting multiplicities);
\item and $\overline{D}$ is the $(d-1)$-plane in $\bP^{n-1}$ spanned by such an effective divisor (notation as in \cite[p.2529, preceding Lemma 2.6]{fis10}).
\end{itemize}

We prove in \cite[Theorem 5.6]{00-leaves-2} (via \cite[Remark 5.7]{00-leaves-2}) that $\Sec_{d,z}\setminus \Sec_{d-1}$ is smooth. While the argument cited above (\cite[p.18, Proof, part 1]{copp}) does not quite go through in this setup, it will if we lower the degree of the locus being removed: this will provide half of \Cref{th:singlocdz}. 

\begin{theorem}\label{th:singlocdz}
  Let $E\subset \bP^{n-1}$ be a normal elliptic curve, $d<\frac n2$ a positive integer, and $z\in E$ a point.

  The singular locus of $\Sec_{d,z}$ is precisely $\Sec_{d-2}$.
\end{theorem}
\begin{proof}
  The claim is that for point $p\in \Sec_{d,z}$
  \begin{equation*}
    p\text{ is singular }\Longleftrightarrow p\in \Sec_{d-2}. 
  \end{equation*}
  We prove the two implications separately.
  \begin{enumerate}[]
  \item {\bf ($\Longleftarrow$)} The already-mentioned \cite[p.18, Proof, part 1]{copp} applies virtually verbatim to show that points in
    \begin{equation*}
      \Sec_{d-2}\subset \Sec_{d,z}
    \end{equation*}
    have full (i.e. $(n-1)$-dimensional) tangent space. Since $\Sec_{d,z}$ is proper in $\bP^{n-1}$ in the regime $2d<n$ under consideration, this proves the desired conclusion.
  \item {\bf ($\Longrightarrow$)} This time around the claim is that $\Sec_{d,z}$ is smooth off $\Sec_{d-2}$. For this, we need to recall some of the setup of \cite[\S 4]{gvb-hul}; \Cref{not:ddnn} is in effect, with
    \begin{equation*}
      n=d+d',
      \quad
      \cL:=\cN\otimes \cN'
      \quad\text{and}\quad
      \cN=\cO(D),\ \sigma(D)=z.
    \end{equation*}
    Identify $\bP^{n-1}\cong \bP H^0(\cL)^*$ and the embedding $E\subset \bP^{n-1}$ with that induced by the full linear system attached to $\cL$. Every element $f\in H^0(\cL)^*$ induces a morphism
    \begin{equation*}
      H^0(\cN)\otimes H^0(\cN')\to H^0(\cL)\xrightarrow{\ f\ }\bC,
    \end{equation*}
    and hence also a linear map
    \begin{equation*}
      H^0(\cN')\to H^0(\cN)^*. 
    \end{equation*}
    Since the various $f$ are precisely the elements of the total bundle $\cO_{\bP^{n-1}}(-1)$, we have a morphism
    \begin{equation}\label{eq:n'n}
      H^0(\cN')\otimes \cO_{\bP^{n-1}}(-1)
      \xrightarrow{\ \phi\ }
      H^0(\cN)^*\otimes \cO_{\bP^{n-1}}
    \end{equation}
    of bundles over $\bP^{n-1}$ (that in \cite[\S 4]{gvb-hul} would have been denoted by $\phi_{\cN}$, making an allowance for the fact that our $\cN$ is that paper's $\cL$). As noted in \cite[discussion preceding Definition 4.2]{gvb-hul} (modulo different notation),
    \begin{itemize}
    \item $\Sec_{d,z}$ is precisely the locus
      \begin{equation*}
        M_{d-1}(\phi):=\{p\in \bP^{n-1}\ |\ \mathrm{rank}(\phi)\le d-1\text{ at }p\};
      \end{equation*}
    \item and hence $\Sec_{d,z}$ is smooth outside the analogously-defined
      \begin{equation*}
        M_{d-2}(\phi):=\{p\in \bP^{n-1}\ |\ \mathrm{rank}(\phi)\le d-2\text{ at }p\}.
      \end{equation*}
    \end{itemize}
    It remains to argue, then, that
    \begin{equation}\label{eq:md2d2}
      M_{d-2}(\phi) = \Sec_{d-2}. 
    \end{equation}
    It follows from \cite[Proposition 4.1]{00-leaves-2} that given a point
    \begin{equation*}
      p\in \bP^{n-1} = \bP H^0(\cL)^*,
    \end{equation*}
    those $s\in H^0(\cN)$ that annihilate the image of \Cref{eq:n'n} at $p$ are precisely those for which
    \begin{equation*}
      p\in \overline{(s)_0} = \text{$(d-1)$-plane spanned by the zero divisor $(s)_0$ of $s$}.
    \end{equation*}
    it follows, then, that the points $p\in M_{d-2}(\phi)$ are precisely those lying on at least two distinct $d$-secants through divisors equivalent to $D$:
    \begin{equation*}
      p\in \mathrm{span}\{z_i,\ 1\le i\le d\}\cap \mathrm{span}\{z'_i,\ 1\le i\le d\}
    \end{equation*}
    for two distinct multisets. It follows from \cite[Lemma 13.2]{gvb-hul} that $p$ belongs to the span of the intersection (counting multiplicities) of the two multisets; that intersection cannot have size $\ge d-1$, because then the condition
    \begin{equation*}
      \sum_i z_i = z = \sum_i z'_i
    \end{equation*}
    would force equality. In conclusion, $p$ belongs to the span of at most $d-2$ points in $E$.

    This proves the `$\subseteq$' inclusion in \Cref{eq:md2d2}, which is sufficient for our purposes. We do, however, have equality in \Cref{eq:md2d2}: for the opposite inclusion `$\supseteq$', simply note that
    \begin{equation*}
      \mathrm{span}\{z_i,\ 1\le i\le d-2\}\subset \bigcap_{z'+z''+\sum z_i=z} \mathrm{span}\{z',\ z'',\ z_i,\ 1\le i\le d-2\}. 
    \end{equation*}
  \end{enumerate}
  This finishes the proof. 
\end{proof}

\section{Remarks on the literature}\label{se:lit}

The above material parallels some of \cite[\S 8]{gvb-hul} in ways that are perhaps instructive to spell out. We once more revert to \Cref{not:ddnn}, with
\begin{equation*}
  \cE := \cN\oplus \cN' = \cO(D)\oplus \cO(D'),
  \quad
  \cL := \cN\otimes \cN'
  \quad\text{and}\quad
  z:=\sigma(D). 
\end{equation*}

We have realized the homological leaf
\begin{equation}\label{eq:homleaf}
  L(\cE)\cong \Sec'_{d,z}:=\Sec_{d,z}\setminus \Sec_{d-1}
\end{equation}
as a divisor complement in
\begin{equation*}
  \widetilde{L}(\cE)\cong \bP_{\beta_{\cN\prec \cN'}}
\end{equation*}
(via \Cref{th:wtlispp}, with $\beta_{\bullet\prec \bullet}$ as in \Cref{def:ellbil}). The same effect is achieved in \cite[\S 8]{gvb-hul} somewhat differently. There, the focus is on the secant varieties themselves (rather than the homological leaves that are central to the present discussion), and on the determinantal varieties $\Sec_{d,z}$ (which is nothing but the $X_{|D|}$ of \cite[p.47]{gvb-hul}).

Said determinantal variety $\Sec_{d,z}$ has a desingularization denoted on \cite[p.47]{gvb-hul} by $\bP(\cE)$ (not at all the same $\cE$ as ours!). The short displayed on \cite[p.47]{gvb-hul} is precisely \Cref{eq:ses} for $\beta_{\cN\prec \cL}$ in place of $\beta_{\cN\prec \cN'}$. Taking into account the fact that loc.cit. uses the opposite convention on projectivizations (i.e. $\bP V$ denotes, there, the space of hyperplanes of $V$ rather than that of lines therein), it follows that the desingularization in question, also denoted by `$\rho$' on \cite[p.47]{gvb-hul}, is
\begin{equation}\label{eq:gvbrho}
  \bP^*_{\beta_{\cN\prec \cL}}\xrightarrow{\quad \rho\quad} \Sec_{d,z}
\end{equation}
(with the asterisk denoting the dual projective bundle). The preimage
\begin{equation}\label{eq:gvbpreim}
  \rho^{-1}(\Sec_{d-1})\subset \bP^*_{\beta_{\cN\prec \cL}}
\end{equation}
is shown in \cite[Corollary 8.6]{gvb-hul} to be anticanonical; furthermore, it is not difficult to see (as in the proof of \cite[Proposition 8.15]{gvb-hul}, say) that the restriction of \Cref{eq:gvbrho} to
\begin{equation*}
  \bP^*_{\beta_{\cN\prec \cL}}\setminus \rho^{-1}(\Sec_{d-1}) 
\end{equation*}
is an isomorphism onto \Cref{eq:homleaf}. This, then, recovers $L(\cE)$ as a(n anticanonical) divisor complement in $\bP^*_{\beta_{\cN\prec \cL}}$ in place of our $\widetilde{L}(\cE)\cong \bP_{\beta_{\cN\prec \cN'}}$.

Of crucial importance in showing, in \cite[Corollary 8.6]{gvb-hul}, that \Cref{eq:gvbpreim} is anticanonical is the computation of the intersection ring of
\begin{equation*}
  \bP^*_{big}:=\bP^*_{\beta_{\cN\prec \cL}}
\end{equation*}
in \cite[Corollary 8.3]{gvb-hul}. That short proof leverages the realization of $\bP^*_{big}$ as a desingularization of a determinantal variety, but it is worthwhile to recover those intersection numbers via Chern-class theory alone: it will then become apparent how to adapt the discussion to the other spaces $\bP:=\bP_{\beta_{\cN\prec \cN'}}$ of interest here.

As noted in the preceding discussion, we have
\begin{equation*}
  \bP^*_{big} = \bP(\cQ_{\beta_{\cN\prec \cL}}^*)
  \quad\text{and}\quad
  \bP = \bP(\cQ_{\beta_{\cN\prec \cN'}})
\end{equation*}
(with $\cQ_{\beta}$ as in \Cref{def:pdpd}). In both cases, the intersection theory depends only on the Chern classes of the bundles being projectivized ($\cQ_{\beta_{\cN\prec \cL}}^*$ and $\cQ_{\beta_{\cN\prec \cN'}}$ respectively).

We assume some of the familiar background on {\it Chow rings} $A(X)$ \cite[\S 1.2]{3264} of (smooth, projective) varieties $X$. Throughout the discussion, we employ

\begin{notation}\label{not:pzeta}
  Let $\cV$ be a rank-$r$ vector bundle over a projective space $\bP^{d-1}$. For the resulting projective bundle
  \begin{equation}\label{eq:pipp}
    \pi:\bP\cV\to \bP^{d-1}
  \end{equation}
  we write
  \begin{itemize}
  \item $\zeta\in A^{1}(\bP\cV)$ for the Chern class of the line bundle $\cO_{\bP\cV}(1)$ on $\bP\cV$ \cite[\S 5.3]{3264};
  \item $h\in A^1(\bP\cV)$ for the Chern class of (the pullback through $\pi$ of) a hyperplane class in
    \begin{equation*}
      A^1(\bP^{d-1})\cong \bZ h.
    \end{equation*}
  \end{itemize}
\end{notation}

By \cite[Theorems 2.1 and 9.6]{3264} the pullback $\pi^*$ is an embedding
\begin{equation*}
  \bZ[h]/(h^{d})\cong A(\bP^{d-1})
  \quad\subset\quad
  A(\bP\cV)\cong A(\bP^{d-1})[\zeta]/(\zeta^r+c_1(\cV)\zeta^{r-1}+\cdots+c_r(\cV)),
\end{equation*}
with
\begin{equation*}
  c_i(\cV)\in A^i(\bP)\cong
  \begin{cases}
    \bZ h^i, &i\le d-1\\
    \{0\}, &i>d-1
  \end{cases}
\end{equation*}
being the Chern classes of $\cV$. \Cref{le:intnrs} below is an immediate consequence; the map
\begin{equation*}
  \deg:A^{\text{top degree}}(X)\to \bZ
\end{equation*}
for a smooth projective variety $X$, referred to implicitly in the statement is that of \cite[Definition 1.4]{Fulton-2nd-ed-98} (or \cite[p.426]{hrt}).

\begin{lemma}\label{le:intnrs}
  Let \Cref{eq:pipp} be a projective $(r-1)$-bundle over $\bP^{d-1}$.
  \begin{enumerate}[(a)]
  \item\label{item:intnrs1} The top-degree component $A^{r+d-2}(\bP\cV)$ of the Chow ring is a copy of $\bZ$, generated by $\zeta^{r-1} h^{d-1}$.
  \item\label{item:intnrs2} For $s\in \bZ_{\ge 0}$ the intersection number $\deg(\zeta^{r-1+s} h^{d-1-s})$ is the coefficient of $\zeta^{r-1}h^{d-1}$ in
    \begin{equation*}
      \left(\text{residue of $\zeta^{r-1+s}$ modulo $\zeta^r+c_1(\cV)\zeta^{r-1}+\cdots+c_r(\cV)$}\right)\cdot h^{d-1-s}.
    \end{equation*}
  \end{enumerate}  
  \qedhere
\end{lemma}

\begin{convention}\label{cv:zhz}
  We identify the homogeneous components
  \begin{equation}\label{eq:aihi}
    A^i(\bP^{d-1})\cong \bZ h^i
  \end{equation}
  with copies of $\bZ$, by distinguishing the natural generator $h^i$ therein. Correspondingly, the Chern classes $c_i(\cV)$ can now be thought of as integers (the respective coefficients of $h^i$ in \Cref{eq:aihi}).
  
  Similarly, for a rank-$r$ bundle $\cV$ over $\bP^{d-1}$ we regard the top component
  \begin{equation*}
    A^{r+d-2}(\bP\cV)\cong \bZ \zeta^{r-1}h^{d-1}
  \end{equation*}
  as $\bZ$, generated by $z^{r-1}h^{d-1}$. This will allow us to drop the `$\deg$' symbol in working with intersection numbers: that of \Cref{le:intnrs} \Cref{item:intnrs2} would now simply be $\zeta^{r-1+s} h^{d-1-s}$. The same convention is observed tacitly in \cite[\S 8]{gvb-hul}.
\end{convention}

For the purpose of carrying out computations effectively, it will be useful to appeal to the following notion, for which we refer, say, to \cite[\S 19]{ms-cc} or \cite[\S 1]{hirz}.

\begin{definition}\label{def:multseq}
  Let $\Bbbk$ be a commutative unital ring (for us, typically $\bC$).
  \begin{enumerate}[(1)]
  \item A {\it multiplicative sequence}
    \begin{equation}\label{eq:seqk}
      {\bf K} = (K_n)_{n\ge 0}
    \end{equation}
    over $\Bbbk$ consists of polynomials
    \begin{equation*}
      K_n(c_1,\cdots,c_n)\in \Bbbk[c_1,\cdots,c_n]
    \end{equation*}
    in variables $c_i$ of respective degrees $i$ such that
    \begin{itemize}
    \item $K_n$ is homogeneous of degree $n$;
    \item $K_0\equiv 1$;
    \item and the map
      \begin{equation}\label{eq:multself}
        \left(1+c_1 t+c_2t^2+\cdots\right)
        \quad
        \xmapsto{\quad}
        \quad
        \sum_{n\ge 0}K_n(c_1,\cdots,c_n)t^n
      \end{equation}
      is multiplicative (and hence an endomorphism) on the group $\Bbbk[[t]]^{\times}_1$ of power series over $\Bbbk[c_i,\ i\ge 1]$ with constant term 1.
    \end{itemize}
  \item We also write the right-hand side of \Cref{eq:multself} as
    \begin{equation}\label{eq:kmap}
      {\bf K}\left(1+c_1 t+c_2t^2+\cdots\right)
      :=
      \sum_{n\ge 0}K_n(c_1,\cdots,c_n)t^n
    \end{equation}
    and refer to the power series
    \begin{equation*}
      \chi_{\bf K}(t):={\bf K}(1+t)
    \end{equation*}
    as the {\it characteristic (power) series} of the multiplicative sequence $(K_n)_n$.
  \end{enumerate}  
\end{definition}

In the sequel we conflate, both notationally and conceptually, the sequence \Cref{eq:seqk} and the endomorphism \Cref{eq:kmap} of the multiplicative group $\Bbbk[[t]]^{\times}_1$ of free-term-1 power series. Furthermore, per \cite[lemma 19.1]{ms-cc} (or \cite[Lemmas 1.2.1 and 1.2.2]{hirz}),
\begin{equation*}
  \left(\text{multiplicative sequence ${\bf K}$} \right)
  \xmapsto{\quad}
  \left(\text{its characteristic series $\chi_{\bf K}$} \right)
\end{equation*}
is a bijection between multiplicative sequences and formal power series with free term 1. 

With all of this in place, we can now revisit \Cref{le:intnrs}.

\begin{proposition}\label{pr:isinv}
  Let $\cV$ be a rank-$r$ bundle over $\bP^{r-1}$ with Chern classes
  \begin{equation*}
    c_i\in A^i(\bP^{r-1})\cong \bZ h^i\cong \bZ,\ 1\le i\le r. 
  \end{equation*}
  On $\bP\cV$ we have, for $s\ge 0$,
  \begin{equation*}
    \zeta^{r-1+s}h^{d-1-s} = H_s(c_1,\cdots,c_s)
  \end{equation*}
  where ${\bf H}=(H_i)$ is the multiplicative sequence with characteristic power series
  \begin{equation*}
    \chi_{\bf H}(t) = (1+t)^{-1} = 1-t+t^2-t^3+\cdots.
  \end{equation*}
\end{proposition}
\begin{proof}
  In view of \Cref{cv:zhz}, the claim is that upon regarding everything in sight as formal variables, the coefficient (denoted below by $[\zeta^{r-1+s}]_{{\bf c},r-1}$) of $\zeta^{r-1}$ in the residue of $\zeta^{r-1+s}$ modulo
  \begin{equation*}
    f_{\bf c}(\zeta):=\zeta^r+c_1\zeta^{r-1}+\cdots+c_r,\ {\bf c} = (c_1,\cdots,c_r)
  \end{equation*}
  is $H_s(c_1,\cdots,c_s)$. To see this, note first (via a simple calculation we forego) that the map sending the monic polynomial $f_{\bf c}$ (of degree $r$) to
  \begin{equation*}
    1 + [\zeta^{r-1+1}]_{{\bf c},r-1}t + [\zeta^{r-1+2}]_{{\bf c},r-1}t^2 + \cdots \in \bC[c_i][[t]]^{\times}_1
  \end{equation*}
  is multiplicative. The conclusion then follows from the fact that that map sends the polynomial
  \begin{equation*}
    f_{(1)}(\zeta) = \zeta+1
  \end{equation*}
  to $(1+t)^{-1}$. 
\end{proof}

\begin{remark}
  The multiplicative sequence ${\bf H}$ of \Cref{pr:isinv} appears, for instance, as \cite[\S 19, Example 2]{ms-cc}.
\end{remark}

\Cref{pr:isinv} in turn implies 

\begin{corollary}\label{cor:invbun}
  Let $\cV$ be a rank-$r$ bundle over $\bP^{r-1}$ and $\cV^*$ its dual bundle. The intersection numbers on the two resulting projective bundles are related by
  \begin{equation*}
    \zeta_{\cV}^{r-1+s}h_{\cV}^{d-1-s}
    =
    (-1)^s\zeta_{\cV^*}^{r-1+s}h_{\cV^*}^{d-1-s},\ s\in \bZ_{\ge 0}.
  \end{equation*}
\end{corollary}
\begin{proof}
  Immediate from \Cref{pr:isinv}, given that the $H_s$ are homogeneous of respective degrees $s$ and
  \begin{equation*}
    c_s(\cV^*) = (-1)^s c_s(\cV)
  \end{equation*}
  by \cite[Remark 3.2.3 (a)]{Fulton-2nd-ed-98}.
\end{proof}

Returning, now, to bundles $\cQ_{\beta}$ (and their projectivizations $\bP_{\beta}$) attached to 1-generic pairings $\beta$:

\begin{proposition}\label{pr:gvbint}
  Let \Cref{eq:bilin} be a 1-generic pairing with
  \begin{equation*}
    d:=\dim V,\ d':=\dim V',\ k:=\dim W. 
  \end{equation*}
  \begin{enumerate}[(a)]
  \item\label{item:gvbint1} On $\bP_{\beta}=\bP\cQ_{\beta}$ we have the intersection numbers
    \begin{equation*}
      \zeta^{r-1+s}h^{d-1-s} = (-1)^s\tbinom{k}{s},\ s\ge 0.
    \end{equation*}
  \item\label{item:gvbint2} On $\bP_{\beta}^*=\bP\cQ_{\beta}^*$ we have the intersection numbers
    \begin{equation*}
      \zeta^{r-1+s}h^{d-1-s} = \tbinom{k}{s},\ s\ge 0.
    \end{equation*}
  \end{enumerate}
\end{proposition}
\begin{proof}
  The two claims are equivalent by \Cref{cor:invbun}, so we focus on part \Cref{item:gvbint1}. Denote
  \begin{equation*}
    c_s := \zeta^{r-1+s}h^{d-1-s}.
  \end{equation*} 
  We will abuse notation and substitute $\zeta$ itself for the formal variable $t$ of \Cref{def:multseq}.

  By \Cref{pr:chern,pr:isinv} we have
  \begin{equation}\label{eq:cik}
    1+c_1\zeta +c_2\zeta^2 +\cdots
    =
    {\bf K}\left((1+\zeta+\zeta^2+\cdots)^k\right)
  \end{equation}
  for the multiplicative sequence ${\bf H}$ of \Cref{pr:isinv}. That multiplicative sequence, regarded as an endomorphism of $\bC[[\zeta]]^{\times}_1$, is nothing but inversion:
  \begin{equation*}
    {\bf K}(f) = f^{-1},\ \forall f\in \bC[[\zeta]]^{\times}_1.
  \end{equation*}
  It follows that \Cref{eq:cik} equals 
  \begin{equation*}
    (1+\zeta+\zeta^2+\cdots)^{-k} = ((1-\zeta)^{-1})^{-k} = (1-\zeta)^k,
  \end{equation*}
  hence the conclusion: the coefficient of $\zeta^s$ therein is precisely $(-1)^s\tbinom{k}{s}$. 
\end{proof}

\begin{remark}
  As discussed above, \Cref{pr:gvbint} \Cref{item:gvbint2} in turn specializes to (and recovers by different means) the intersection numbers of \cite[Corollary 8.3]{gvb-hul}.
\end{remark}

We also record the following analogue of (the first part of) \cite[Proposition 8.4]{gvb-hul}.

\begin{lemma}\label{le:can}
  Under the hypotheses of \Cref{pr:gvbint}, and using \Cref{not:pzeta} with $r=d$, the canonical divisor on $\bP_{\beta}$ is $-d\zeta-d'h$. 
\end{lemma}
\begin{proof}
  The argument is entirely parallel to that proving the first part of \cite[Proposition 8.4]{gvb-hul}, taking into account the mutual duality between the conventions on projectivization.

  Consider the projection
  \begin{equation*}
    \bP_{\beta} = \bP\cQ_{\beta}\xrightarrow{\quad \pi\quad} \bP^{d-1} = \bP V. 
  \end{equation*}
  By \cite[Exercise III.8.4]{hrt} (and taking into account said duality), we have the expression 
  \begin{equation}\label{eq:reldual}
    \omega_{\bP_{\beta}}\otimes \pi^*\omega_{\bP^{d-1}}^{-1}\cong \left(\pi^*\cQ_{\beta}^*\right)(-d)
  \end{equation}
  for the {\it relative canonical sheaf} attached to $\pi$. Dualizing \Cref{pr:chern} and substituting $h$ for $\zeta$ (as appropriate with the current notation conventions), the total Chern class of $\cQ_{\beta}^*$ is
  \begin{equation*}
    c(\cQ_{\beta}^*) = (1-h)^k = (1-h)^{d'-d}.
  \end{equation*}
  It follows that the divisor class of \Cref{eq:reldual} is $-d\zeta-k h$. Adding to that the divisor class $-d h$ of the canonical pullback $\pi^*\omega_{\bP^{d-1}}^{-1}$ \cite[Example II.8.20.1]{hrt}, we obtain the desired formula:
  \begin{align*}
    \omega_{\bP_{\beta}} &\cong \left(\omega_{\bP_{\beta}}\otimes \pi^*\omega_{\bP^{d-1}}^{-1}\right)\otimes \pi^*\omega_{\bP^{d-1}}^{-1}\\
                         &\rightsquigarrow (-d\zeta-k h)+(-d h)\\
                         &=-d\zeta-d' h,
  \end{align*}
  with `$\rightsquigarrow$' meaning `has divisor class'. 
\end{proof}

\begin{remark}
  As noted before, \Cref{le:can} is entirely analogous to the first part of \cite[Proposition 8.4]{gvb-hul}: there, the coefficients are $-d$ and $n-2d$, where
  \begin{itemize}
  \item the $n$ of loc.cit. plays the same role as our $d'$;
  \item so that $n-d$ is our $k$;
  \item and the quantity $n-d$ is {\it added} to a $-d$ (as opposed to being subtracted, as in our case) because we projectivize bundles dually as compared to both \cite{gvb-hul} and \cite[Definition preceding Proposition II.7.11]{hrt}.
  \end{itemize}
\end{remark}

We now return to \Cref{th:ynorm}, finalizing the one outstanding point.

\begin{proposition}\label{pr:anticfinal}
  In the context of \Cref{th:ynorm}, the divisor $Y$ of \Cref{eq:divc} is anticanonical in $\widetilde{L}(\cE)\cong \bP_{\beta_{\cN\prec\cN'}}$.
\end{proposition}
\begin{proof}
  According to \Cref{le:can} (and with \Cref{not:pzeta} in scope), we have to show that the two coefficients in the expression
  \begin{equation}\label{eq:abcoeff}
    Y=a\zeta + bh\in A^1(\bP_{\beta}),\ a,b\in \bZ
  \end{equation}
  are $a=d$ and $b=d'$. The claim, then, is twofold.

  \begin{enumerate}[(a)]
  \item\label{item:aisd} {\bf The coefficient $a$ of \Cref{eq:abcoeff} is $d$.} Fix an element
    \begin{equation}\label{eq:ztupd}
      {\bf z} = (z_1,\ \cdots,\ z_d)\in \bP^{d-1} = \bP H^0(\cN),
    \end{equation}
    so that $\sum z_i=\sigma(D)$. If the $z_i$ are distinct, the fiber of
    \begin{equation}\label{eq:pbp}
      \pi:\bP_{\beta}\to \bP^{d-1}
    \end{equation}
    above ${\bf z}$ consists of the $d$ hyperplanes $H^0(\cN')_{z_i}$, $1\le i\le d$ (in the notation of \Cref{eq:hzs}). This gives 
    \begin{align*}
      d &= Y\zeta^{d-2} h^{d-1}\\
        &= (a\zeta+bh) \zeta^{d-2}h^{d-1}\\
        &= a \quad\text{by \Cref{pr:gvbint} \Cref{item:gvbint1}},
    \end{align*}
    finishing the proof.
    
  \item {\bf The coefficient $b$ of \Cref{eq:abcoeff} is $d'$.} Consider a section $\iota:\bP^{d-1}\to \bP_{\beta}$ of \Cref{eq:pbp} constructed as follows: choose an element $s\in H^0(\cN')$, sufficiently generic so as to ensure that
    \begin{equation*}
      s\not\in H^0(\cN')_{\left(z_i,\ 1\le i\le d\right)}
    \end{equation*}
    for arbitrary $\sum z_i=\sigma(D)$ (notation, again, as in \Cref{eq:hzs}). The section $\iota=\iota_s$ is then
    \begin{equation*}
      \bP^{d-1}=\bP H^0(\cN) (z_1,\ \cdots,\ z_d)\xmapsto{\quad \iota\quad} H^0(\cN')_{\left(z_i,\ 1\le i\le d\right)}\oplus \bC s. 
    \end{equation*}
    The image
    \begin{equation*}
      S=S_s = \iota_s(\bP^{d-1})\subset \bP_{\beta}
    \end{equation*}
    is a codimension-$(d-1)$ subvariety, and hence can be regarded as an element of the Chow group
    \begin{equation*}
      A^{d-1}(\bP^{d-1})\cong \bigoplus_{a+b=d-1} \bZ \zeta^a h^b. 
    \end{equation*}
    The rest of the proof consists of a number of auxiliary claims.
    \begin{enumerate}[(1)]
    \item {\bf We have
        \begin{equation}\label{eq:sz0}
          S\zeta=0\text{ in the Chow ring }A(\bP_{\beta}).
        \end{equation}
      }

      The section $s\in H^0(\cN')$ chosen in the construction of $S=S_s$ provides a trivial line subbundle of $\cQ_{\beta}$: the one whose fiber is the image of $\bC s$ through the right-hand map in the $\cQ_{\beta}$-defining short exact sequence
      \begin{equation*}
        0\to \cP_{\bP^{d-1}}(-1)\otimes H^0(\cN)\to \cP_{\bP^{d-1}}\otimes H^0(\cN')\to \cQ_{\beta}\to 0
      \end{equation*}
      (itself an instance of \Cref{eq:embvb}).

      $S$ is then nothing but the projectivization of that subbundle
      \begin{equation}\label{eq:trivinq}
        \cO_{\bP^{d-1}}\subset \cQ_{\beta}
      \end{equation}
      and the embedding $S\subset \bP_{\beta}$ is precisely the projectivization of the bundle inclusion \Cref{eq:trivinq}. Because the subbundle in \Cref{eq:trivinq} is {\it trivial}, the pullback through $S\subset \bP^{\beta}$ of $\cO_{\bP_{\beta}}(1)$ is also trivial; this follows, for instance, from \cite[Proposition II.7.12]{hrt}, taking into account the fact that that reference's quotient bundles are our subbundles due to convention differences.
      
      We are now done: intersecting $S$ and $\zeta$ (the divisor class of $\cO_{\bP_{\beta}}(1)$) means pulling back $\cO_{\bP_{\beta}}(1)$ to $S$ \cite[\S 2.3]{Fulton-2nd-ed-98}; if the latter produces a trivial bundle, the former must be a trivial intersection.

    \item {\bf We have
        \begin{equation}\label{eq:syd'}
          SY = d'h_{\bP^{d-1}}\text{ in }A(\bP^{d-1})\xrightarrow[\cong]{\quad\iota\quad}A(S),
        \end{equation}
        where $h_{\bP^{d-1}}\in A^1(\bP^{d-1})$ is the hyperplane class. 
      }

      Suppose the section $s\in H^0(\cN')$ chosen in the construction of $\iota=\iota_s$ and $S=S_s=\iota(\bP^{d-1})$ vanishes at the distinct points $z'_j$, $1\le j\le d'$. The pullback through
      \begin{equation*}
        \iota:\bP^{d-1}\to S\subset \bP_{\beta}
      \end{equation*}
      of the intersection $S\cap Y$ then consists precisely of those tuples \Cref{eq:ztupd} for which one of the $z_i$ equals one of the $z'_j$. As there are $d'$ choices for the latter, the intersection in question is a union of $d'$ hyperplanes.
      
    \item {\bf Finishing the proof that $b=d'$.} We have
      \begin{align*}
        d'h_{\bP^{d-1}} &= SY
                          \quad\text{by \Cref{eq:syd'}}\\
                        &= S(d\zeta+bh)
                          \quad\text{by part \Cref{item:aisd} of the present proof}\\
                        &=bSh
                          \quad\text{by \Cref{eq:sz0}}\\
                        &=bh_{\bP^{d-1}},
                          \quad\text{$S$ being the image of a section of \Cref{eq:pbp}.}
      \end{align*}
      Consequently, $b=d'$.
    \end{enumerate}
  \end{enumerate}
  This concludes the proof of the proposition. 
\end{proof}


\addcontentsline{toc}{section}{References}

\def\cprime{$'$}

\Addresses

\end{document}